\documentclass[11pt]{amsart}
\usepackage{amssymb, amsmath, amsthm, mathrsfs, amsopn}
\usepackage{amsrefs}

\usepackage{graphicx}
\usepackage{color}
\usepackage[a4paper, centering]{geometry}
\def\polar{\theta}
\def\nuint{\widetilde{\nu}}

\newcommand*{\jump}[1]{\Theta_{#1}}

\newcommand{\A}{\boldsymbol{A}}
\newcommand{\B}{\boldsymbol B}
\newcommand{\bsmall}{\boldsymbol b}
\newcommand{\F}{F}
\newcommand{\f}{f}

\newcommand{\mean}[1]{\,-\hskip-1.08em\int_{#1}} 

\def\R{\mathbb{R}}

\def\N{\mathbb{N}}
\def\Z{\mathbb{Z}}

\def\DM{{\mathcal{DM}^{\infty}}}
\newcommand*{\DMloc}[1][\Omega]{\mathcal{DM}^{\infty}_{\rm loc}{(#1)}}
\newcommand*{\BVLloc}[1][\Omega]{BV_{{\rm loc}}(#1)\cap L^{\infty}_{{\rm loc}}{(#1)}}
\DeclareMathOperator{\Tr}{Tr}
\newcommand*{\Trace}[3][\pm]{\Tr^{#1}(#2, #3)}
\newcommand*{\Trp}[2]{\Trace[+]{#1}{#2}}
\newcommand*{\Trm}[2]{\Trace[-]{#1}{#2}}

\def\v{{\bf v}}

\def\Xi{\chi^*_{[a_i, a_{i+1}]}}

\DeclareMathOperator{\Div}{div}

\newcommand{\defeq} {:=}
\newcommand{\medint}{-\kern  -,375cm\int}
\newcommand{\medintinrigo}{-\kern  -,315cm\int}

\newcommand{\eps}{\varepsilon}
 \newcommand{\hh}{{\mathcal H}^{N-1}}

\newcommand{\LLN}{{\mathcal L}^N}

\newcommand{\M}[1]{\mathcal{#1}}    
\renewcommand{\H}{\M{H}}

\newcommand{\res}{\mathop{\hbox{\vrule height 7pt width .5pt depth 0pt
\vrule height .5pt width 6pt depth 0pt}}\nolimits} 

\def\pscal#1#2{\left\langle #1\,,\, #2 \right\rangle}
\DeclareMathOperator{\sign}{sign}
\def\ut{\widetilde{u}}

\newtheorem{definition}{Definition}[section]

\newtheorem{lemma}[definition]{Lemma}

\newtheorem{theorem}[definition]{Theorem}

\newtheorem{proposition}[definition]{Proposition}

\newtheorem{corollary}[definition]{Corollary}
\theoremstyle{remark}
\newtheorem{remark}[definition]{Remark}

\makeatletter
\def\@settitle{\begin{center}%
		\baselineskip14\p@\relax
		\bfseries
		\uppercasenonmath\@title
		\@title
		\ifx\@subtitle\@empty\else
		\\[5ex]
		\normalsize\mdseries\@subtitle
		\fi
	\end{center}%
}
\def\subtitle#1{\gdef\@subtitle{#1}}
\def\@subtitle{}
\makeatother

\begin{document}
\title[An extension of the pairing theory]
{An extension of the pairing theory \\ between divergence--measure fields and BV functions}

\author[G.~Crasta]{Graziano Crasta}
\address{Dipartimento di Matematica ``G.\ Castelnuovo'', Univ.\ di Roma I\\
	P.le A.\ Moro 2 -- I-00185 Roma (Italy)}
\email{crasta@mat.uniroma1.it}
\author[V.~De Cicco]{Virginia De Cicco}
\address{Dipartimento di Scienze di Base  e Applicate per l'Ingegneria, Univ.\ di Roma I\\
	Via A.\ Scarpa 10 -- I-00185 Roma (Italy)}
\email{virginia.decicco@sbai.uniroma1.it}

\keywords{Anzellotti's pairing, divergence--measure fields,
Gauss--Green formula}
\subjclass[2010]{28B05,46G10,26B30}

\date{April 15, 2018}

\begin{abstract}
In this paper we introduce a nonlinear version of the notion of Anzellotti's pairing
between divergence--measure vector fields and functions of bounded variation,
motivated by possible applications to
evolutionary quasilinear problems.
As a consequence of our analysis, we prove a generalized Gauss--Green formula.
\end{abstract}

\maketitle

\section{Introduction}

In recent years the pairing theory between divergence--measure vector fields and BV functions,
initially developed by Anzellotti \cite{Anz,Anz2},
has been extended to more general situations
(see e.g.\ \cite{ChenFrid,ChFr1,ChTo2,ChTo,ChToZi,ComiPayne,CDC3,LeoSar2,SchSch,SchSch2}
and the references therein).
These extensions are motivated, among others, 
by applications to 
hyperbolic conservation laws and transport equations \cite{AmbCriMan,ADM,ChenFrid,ChTo2,ChTo,ChToZi,CDC1,CDC2,CDD,CDDG}, 
problems involving the $1$-Laplace operator
\cite{AVCM,K1},
the prescribed mean curvature problem
\cite{LeoSar,LeoSar2}
and to lower semicontinuity problems in BV \cite{BouDM,DCL,DCFV2}.

Another major related result
concerns the Gauss--Green formula
and its applications
(see e.g.\ \cite{Anz,Cas,ComiPayne,CoTo,CDC3,LeoSar2}).

Let us describe the problem in more details.
Let $\DM$ denote the class of
bounded divergence--measure vector fields \(\A\colon\R^N \to \R^N\), 
i.e.\ the vector fields with the properties 
that $\A$ is bounded and
\(\Div \A\) is a finite Radon measure.
If \(\A\in\DM\) and \(u\) is a function of bounded variation with precise representative \(u^*\), 
then Chen and Frid \cite{ChenFrid} proved that
\[
\Div(u\A) = u^* \Div\A + \mu,
\]
where $\mu$ is a Radon measure, absolutely continuous
with respect to $|Du|$.
This measure $\mu$ has been denoted by Anzellotti \cite{Anz}
with the symbol $(\A, Du)$ and by 
Chen and Frid with $\overline{\A\cdot Du}$,
and it is called the pairing between the divergence--measure
field $\A$ and the gradient of the $BV$ function $u$.
The characterization of the decomposition of this measure into absolutely continuous,
Cantor and jump parts has been
studied in \cite{ChenFrid,CDC3}.
In particular, the analysis of the jump part
has been considered in \cite{CDC3} using the notion
of weak normal traces of a divergence--measure vector
field on oriented countably $\H^{N-1}$-rectifiable sets
given in \cite{AmbCriMan}.

In view to applications to evolutionary quasilinear problems,
our aim is to extend the pairing theory
from the product $u\A$
to the mixed case $\B(x, u)$.
Our main assumptions on $\B$ are
that $\B(\cdot, w)\in\DM$ for every $w\in\R$,
$\B(x, \cdot)$ is of class $C^1$
and
the least upper bound
\[
\sigma\defeq \bigvee_{w\in \R} |\Div_x \partial_w\B(\cdot,w)|
\]
is a Radon measure.
(See Section \ref{ss:assumptions} for the complete
list of assumptions.)
We remark that, in the case $\B(x,w) = w\, \A(x)$,
with $\A\in\DM$,
these assumptions are automatically satisfied
with $\sigma = |\Div\A|$. 
More precisely, 
we will prove that,
if $u\in BV\cap L^\infty$, then
the composite function
$\v(x) := \B(x,u(x))$ belongs to $\DM$
and
\[
\Div[\B(x,u(x))]
= 
\frac{1}{2}\left[
F(x, u^+(x)) + F(x, u^-(x))
\right]\, \sigma
+ \mu,
\]
in the sense of measures.
Here
$F(\cdot, w)$ denotes the Radon--Nikod\'ym derivative
of the measure $\Div_x \B(\cdot, w)$
with respect to $\sigma$,
$u^\pm$ are the approximate limits of $u$
and $\mu \equiv (\partial_w \B(\cdot , u), Du)$ is again a Radon measure, absolutely
continuous with respect to $|Du|$
(see Theorem~\ref{chainb4}).
We recall that the analogous chain rule for BV vector fields has been
proved in \cite{ACDD}.
Notice that, when $\B(x, w) = w\, \A(x)$, then $\mu = (\A,Du)$
is exactly the Anzellotti's pairing between $\A$ and $Du$.
Even for general vector fields $\B(x,w)$
we can prove the following characterization of the decomposition 
$\mu=\mu^{ac}+\mu^c+\mu^j$
of this measure into absolutely continuous,
Cantor and jump parts
(see Theorem~\ref{chainb5}):
\begin{gather*}
\mu^{ac}=\pscal{\partial_w\B(x, u(x))}{\nabla u(x)}\mathcal L^N,
\\
\mu^{c}=\pscal{\widetilde{\partial_w\B}(x, \ut(x))}{D^c u}\,,
\\
\mu^j=
[\beta^*(x,u^+(x))
-\beta^*(x,u^-(x))]
\hh\res {J_u},
\end{gather*}
where, for every $w\in\R$, $\beta^\pm(\cdot, w)$ are the normal traces of $\B(\cdot, w)$
on $J_u$ and $\beta^*(\cdot, w) := [\beta^+(\cdot, w) + \beta^-(\cdot, w)]/2$,
and $\mu^c$ is 
the Cantor part of the measure.
We remark that, to prove the representation formula for $\mu^c$,
we need an additional technical assumption (see \eqref{f:assumpDc} in Theorem \ref{chainb5}).

We recall that a similar characterization
of $\Div[\B(x,u(x))]$ has been proved
in \cite{CDC2} under stronger assumptions,
including the existence of the strong traces
of $\B(\cdot, w)$.

As a consequence of our analysis 
we prove that, if
\(E \subset \R^N\) is a bounded set with finite perimeter,
then the following Gauss--Green formula holds:
\[
	\int_{E^1} \frac{ F(x, u^+(x)) + F(x, u^-(x)) }{2}\, d\sigma(x)
	+ \mu(E^1) =
	-\int_{\partial ^*E} \beta^+(x,u^+(x)) \ d\mathcal H^{N-1}\,,
\]
where $E^1$ is the measure theoretic interior of $E$,
and $\partial^* E$ is the reduced boundary of $E$
(see Theorem \ref{t:GG}).
In the particular case $\B(x, w) = w\, \A(x)$, the
analogous formula has been proved in
\cite[Theorem 5.1]{CDC3}. 
We recall that the map $x\mapsto \beta^+(x,u^+(x))$
in the last integral coincides with the interior weak normal
trace of the vector field $x\mapsto \B(x, u(x))$
on $\partial^* E$
(see Proposition~\ref{p:traces}).

\medskip
The plan of the paper is the following.
In Section~\ref{s:prel} we recall some known results 
on functions of bounded variation, 
divergence--measure vector fields
and their normal traces,
and the Anzellotti's pairing.

In Section~\ref{s:div}
we list the assumptions on $\B$ and
we prove a number of its basic properties,
including some regularity result of the weak normal traces.
Finally, we prove that for every $u\in BV\cap L^\infty$,
the composite function $x\mapsto \B(x, u(x))$ belongs to
$\DM$.

In Section \ref{ss:div2} we prove our main results
on the pairing measure $\mu$ and its properties.

Finally, in Section \ref{s:gluing} we describe
some gluing construction and we prove an extension theorem
that will be used in Section~\ref{s:green}
to prove a Gauss--Green formula for weakly
regular domains.

\section{Preliminaries}\label{s:prel}

In the following, \(\Omega\) will always denote a nonempty open subset of \(\R^N\).

Let \(u\in L^1_{{\rm loc}}(\Omega)\).
We say that \(u\) has an approximate limit at $x_{0}\in\Omega$ 
if there exists \(z\in\R\) such that
\begin{equation}
	\label{f:apcont}
	\lim_{r\rightarrow0^{+}}\frac{1}{\LLN\left(  B_r(x_0)\right)}\int_{B_r\left(  x_{0}\right)
	}\left|  u(x)  -z  \right|  \,dx=0.
\end{equation}
The set \(S_u\subset\Omega\) of points where this property does not hold is called
the approximate discontinuity set of \(u\).
For every \(x_0\in\Omega\setminus S_u\), the number \(z\), uniquely determined by
\eqref{f:apcont}, is called the approximate limit of \(u\) at \(x_0\)
and denoted by \(\ut(x_0)\).

We say that \(x_0\in\Omega\) is an approximate jump point of \(u\) if
there exist \(a,b\in\R\) and a unit vector \(\nu\in\R^n\) such that \(a\neq b\)
and
\begin{equation}\label{f:disc}
	\begin{gathered}
		\lim_{r \to 0^+} \frac{1}{\LLN(B_r^+(x_0))}
		\int_{B_r^+(x_0)} |u(y) - a|\, dy = 0,
		\\
		\lim_{r \to 0^+} \frac{1}{\LLN(B_r^-(x_0))}
		\int_{B_r^-(x_0)} |u(y) - b|\, dy = 0,
	\end{gathered}
\end{equation}
where \(B_r^\pm(x_0) := \{y\in B_r(x_0):\ \pm (y-x_0)\cdot \nu > 0\}\).
The triplet \((a,b,\nu)\), uniquely determined by \eqref{f:disc} 
up to a permutation
of \((a,b)\) and a change of sign of \(\nu\),
is denoted by \((u^+(x_0), u^-(x_0), \nu_u(x_0))\).
The set of approximate jump points of \(u\) will be denoted by \(J_u\).

The notions of approximate discontinuity set, approximate limit and approximate jump point 
can be obviously extended to the vectorial case
(see \cite[\S 3.6]{AFP}).

In the following we shall always extend the functions $u^\pm$ to
\(\Omega\setminus(S_u\setminus J_u)\) by setting
\[
u^\pm \equiv \widetilde{u}\ \text{in}\ \Omega\setminus S_u.
\]

\begin{definition}[Strong traces]\label{def:tracce}
	Let $u\in L^\infty_{\rm loc}(\R^{N})$ and  
	let $\mathcal J\subset \R^{N}$ be a countably $\H^{N-1}$-rectifiable set oriented by a normal vector
	field $\nu$. 
	We say that two Borel functions \(u^{\pm}\colon\mathcal{J}\to\R\)
	are the strong traces of \(u\) on \(\mathcal{J}\)
	if for \(\H^{N-1}\)-almost every \(x\in\mathcal{J}\)
	it holds
	\[
	\lim_{r\to 0^+} \int_{B_r^\pm(x)} |u(y) - u^\pm(x)|\, dy = 0,
	\]
	where
	$B_r^\pm(x) :=  B_r(x) \cap \{ y\in\R^{N} : \pm \langle y-x,\nu(x)\rangle \geq 0 \}$.
\end{definition}

Here and in the following we will denote by \(\rho \in C^\infty_c(\R^N)\) a
symmetric convolution kernel with support in the unit ball,
and by \(\rho_{\varepsilon}(x) := \varepsilon^{-N} \rho(x/\varepsilon)\).

In the sequel we will use often the following result.

\begin{proposition}\label{lolol}
	Let \(u\in L^1_{{\rm loc}}(\Omega)\)
	and define
	\[
	u_{\varepsilon}\left(  x\right) = \rho_{\varepsilon} \ast u (x) :=\int_{\Omega}\rho_{\varepsilon}\left(
	x-y\right)  \,u\left(  y\right)  \,dy.
	\]
	If \(x_0\in\Omega\setminus S_u\), then \(u_\eps(x_0) \to \ut(x_0)\)
	as \(\eps\to 0^+\).
	%
	%
\end{proposition}

\begin{proposition} 
	\label{p:ScorzaDragoni}
	Let $E$ be a Lebesgue measurable subset of ${\mathbb{R}}^{N}$ 
	and $G$ a Borel subset of ${\mathbb{R}}^M$. 
	Let $g\colon E\times G\rightarrow\mathbb{R}$ be a Borel function such that 
	for $\mathcal{L}^{N}$--a.e.\ $x\in$ $E$ the function 
	$g\left(  x,\cdot\right)$ is continuous on $G$.
	Then there exists an $\mathcal{L}^{N}$-null set $\mathcal{M}\subset
	\mathbb{R}^{N}$ such that for every $t\in G$ the function
	$g(\cdot,t)$ is approximately continuous in $E\setminus\mathcal{M}$.
\end{proposition}

\begin{proof}
	(See the proof of Theorem~1.3, p.~539 in \cite{FonLeo}.)
	By the Scorza--Dragoni theorem
	(see \cite[Theorem 6.35]{FonLeoBook}), for every \(i\in\Z^+\) there exists
	a closed set \(K_i \subset E\), with \(\LLN(E\setminus K_i) < 1/i\), 
	such that the restriction of \(g\) to \(K_i\times G\) is continuous.
	Let \(K_i^*\) be the set of 
	points with density \(1\) of \(K_i\),
	and define
	\[
	\mathcal{E} := \bigcup_{i=1}^\infty (K_i \cap K_i^*).
	\]
	Since \(\LLN(K_i) = \LLN(K_i\cap K_i^*)\), it holds
	\[
	\LLN(E \setminus \mathcal{E}) \leq \LLN(E \setminus K_i) \leq \frac{1}{i}
	\qquad \forall i\in \Z^+,
	\]
	and so \(\LLN(E \setminus \mathcal{E}) = 0\).
	
	Let us fix \(t\in G\) and let us prove that the function
	\(u(x) := g(x, t)\) is 
	approximately continuous on \(\mathcal{E}\).
	Let \(\varepsilon > 0\) and \(x_0\in \mathcal{E}\).
	By definition of \(\mathcal{E}\), there exists an index
	\(i\in\Z^+\) such that \(x_0 \in K_i \cap K_i^*\).
	Since the restriction of \(u\) to \(K_i\) is continuous,
	there exists \(\delta > 0\) such that
	\[
	|u(x) - u(x_0)| < \varepsilon
	\qquad
	\forall x \in B_\delta(x_0) \cap K_i.
	\]
	As a consequence, for every \(r \in (0, \delta)\), it holds
	\[
	\LLN(B_r(x_0) \cap \{x\in E:\ |u(x) - u(x_0)| \geq \varepsilon\})
	\leq \LLN(B_r(x_0) \setminus K_i),
	\]
	hence the conclusion follows since \(x_0\) is a Lebesgue point of \(K_i\).
\end{proof}

\begin{corollary}\label{mmmmm}
	Let $E$ be a Lebesgue measurable subset of ${\mathbb{R}}^{N}$ 
	and $G$ a Borel subset of $\mathbb{R}^M$.
	Let $g\colon E\times G\rightarrow\mathbb{R}$ be a Borel function such that 
	for $\mathcal{L}^{N}$--a.e.\ $x\in$ $E$ the function $g\left(  x,\cdot\right)  $ is continuous on $G$.
	Let us define
	\[
	g_{\varepsilon}\left(  x, t\right)  :=\int_{\R^N}\rho_{\varepsilon}\left(
	x-y\right)  \,g\left(  y,t\right)  \,dy.
	\]
	Then there exists an $\mathcal{L}^{N}$-null set 
	$Z\subset\mathbb{R}^{N}$ such that for every $t\in G$ and for every $x\in E\setminus Z$ 
	we have
	$g_{\varepsilon}\left(  x_{0},t\right)  \rightarrow g\left(  x_{0},t\right)  $, as
	$\varepsilon\rightarrow0^{+}$.
\end{corollary}

\subsection{Functions of bounded variation and sets of finite perimeter}

We say that \(u\in L^1(\Omega)\) is a function of bounded variation in \(\Omega\)
if the distributional derivative \(Du\) of \(u\) is a finite Radon measure in \(\Omega\).
The vector space of all functions of bounded variation in \(\Omega\)
will be denoted by \(BV(\Omega)\).
Moreover, we will denote by \(BV_{{\rm loc}}(\Omega)\) the set of functions
\(u\in L^1_{{\rm loc}}(\Omega)\) that belong to 
\(BV(A)\) for every open set \(A\Subset\Omega\)
(i.e., the closure \(\overline{A}\) of \(A\) is a compact
subset of \(\Omega\)).

If \(u\in BV(\Omega)\), then \(Du\) can be decomposed as
the sum of the absolutely continuous and the singular part with respect
to the Lebesgue measure, i.e.\
\[
Du = D^a u + D^s u,
\qquad D^a u = \nabla u \, \LLN,
\]
where \(\nabla u\) is the approximate gradient of \(u\),
defined \(\LLN\)-a.e.\ in \(\Omega\).
On the other hand, the singular part \(D^s u\) can be further decomposed
as the sum of its Cantor and jump part, i.e.
\[
D^s u = D^c u + D^j u,
\qquad
D^c u := D^s u \res (\Omega\setminus S_u),
\quad
D^j u := D^s u \res J_u,
\]
where the symbol \(\mu\res B\) denotes the restriction of the measure \(\mu\)
to the set \(B\).
We will denote by \(D^d u := D^a u + D^c u\) the diffuse part of the measure \(Du\).

In the following,
we will denote by $\polar_u\colon\Omega\to S^{N-1}$ the Radon--Nikod\'ym derivative
of $Du$ with respect to $|Du|$, i.e.\
the unique function $\polar_u \in L^1(\Omega, |Du|)^N$ such that the polar
decomposition $Du = \polar_u\, |Du|$ holds.
Since all parts of the derivative of $u$ are mutually singular, we have
\[
D^a u = \polar_u\, |D^a u|,
\quad
D^j u = \polar_u\, |D^j u|,
\quad
D^c u = \polar_u\, |D^c u|
\]
as well.
In particular $\polar_u(x) = \nabla u(x) / |\nabla u(x)|$ for
$\LLN$-a.e.\ $x\in\Omega$ such that $\nabla u(x) \neq 0$
and
$\polar_u(x) = \sign(u^+(x) - u^-(x))\, \nu_u(x)$
for $\H^{N-1}$-a.e.\ $x\in J_u$.

Let \(E\) be an \(\LLN\)-measurable subset of \(\R^N\).
For every open set \(\Omega\subset\R^N\) the perimeter \(P(E, \Omega)\)
is defined by
\[
P(E, \Omega) := \sup\left\{
\int_E \Div \varphi\, dx:\ \varphi\in C^1_c(\Omega, \R^N),\ \|\varphi\|_\infty\leq 1
\right\}.
\]
We say that \(E\) is of finite perimeter in \(\Omega\) if \(P(E, \Omega) < +\infty\).

Denoting by \(\chi_E\) the characteristic function of \(E\),
if \(E\) is a set of finite perimeter in \(\Omega\), then
\(D\chi_E\) is a finite Radon measure in \(\Omega\) and
\(P(E,\Omega) = |D\chi_E|(\Omega)\).

If \(\Omega\subset\R^N\) is the largest open set such that \(E\)
is locally of finite perimeter in \(\Omega\),
we call reduced boundary \(\partial^* E\) of \(E\) the set of all points
\(x\in \Omega\) in the support of \(|D\chi_E|\) such that the limit
\[
\nuint_E(x) := \lim_{\rho\to 0^+} \frac{D\chi_E(B_\rho(x))}{|D\chi_E|(B_\rho(x))}
\]
exists in \(\R^N\) and satisfies \(|\nuint_E(x)| = 1\).
The function \(\nuint_E\colon\partial^* E\to S^{N-1}\) is called
the measure theoretic unit interior normal to \(E\).

A fundamental result of De Giorgi (see \cite[Theorem~3.59]{AFP}) states that
\(\partial^* E\) is countably \((N-1)\)-rectifiable
and \(|D\chi_E| = \hh\res \partial^* E\).

Let \(E\) be an \(\LLN\)-measurable subset of \(\R^N\).
For every \(t\in [0,1]\) we denote by \(E^t\) the set
\[
E^t := \left\{x\in\R^N:\
\lim_{\rho\to 0^+} \frac{\LLN(E\cap B_\rho(x))}{\LLN(B_\rho(x))} = t\right\}
\]
of all points where \(E\) has density \(t\).
The sets \(E^0\), \(E^1\), \(\partial^e E := \R^N\setminus (E^0 \cup E^1)\) are called 
respectively the measure theoretic exterior, the measure theoretic interior and
the essential boundary of \(E\).
If \(E\) has finite perimeter in \(\Omega\), Federer's structure theorem
states that
\(\partial^* E\cap\Omega \subset E^{1/2} \subset \partial^e E\)
and \(\H^{N-1}(\Omega\setminus(E^0\cup \partial^e E \cup E^1)) = 0\)
(see \cite[Theorem~3.61]{AFP}).

\subsection{Divergence--measure fields }
\label{ss:div}

We will denote by \(\DM(\Omega)\) the space of all
vector fields 
\(\A\in L^\infty(\Omega, \R^N)\)
whose divergence in the sense of distributions is a bounded Radon measure in \(\Omega\).
Similarly, \(\DMloc[\Omega]\) will denote the space of
all vector fields \(\A\in L^\infty_{{\rm loc}}(\Omega, \R^N)\)
whose divergence in the sense of distributions is a Radon measure in \(\Omega\). 
We set \(\DM = \DM(\R^N)\).

We recall that, if \(\A\in\DMloc[\Omega]\), then \(|\Div\A| \ll \hh\)
(see \cite[Proposition 3.1]{ChenFrid}).
As a consequence, the set
\begin{equation}\label{f:jump}
	\jump{\A} 
	:= \left\{
	x\in\Omega:\
	\limsup_{r \to 0+}
	\frac{|\Div \A| (B_r(x))}{r^{N-1}} > 0
	\right\},
\end{equation} 
is a Borel set, \(\sigma\)-finite with respect to \(\hh\),
and the measure \(\Div \A\) can be decomposed as
\[
\Div\A = \Div^a\A + \Div^c\A + \Div^j\A,
\]
where \(\Div^a\A\) is absolutely continuous with respect to \(\LLN\),
\(\Div^c\A (B) = 0\) for every set \(B\) with \(\hh(B) < +\infty\),
and
\[
\Div^j\A = h\, \hh\res\jump{\A}
\]
for some Borel function \(h\)
(see \cite[Proposition~2.5]{ADM}).

\subsection{Anzellotti's pairing}
As in Anzellotti \cite{Anz} (see also \cite{ChenFrid}),
for every \(\A\in\DMloc\) and \(u\in\BVLloc\) we define the linear functional
\((\A, Du) \colon C^\infty_0(\Omega) \to \R\) by
\begin{equation}\label{f:pairing}
	\pscal{(\A, Du)}{\varphi} :=
	-\int_\Omega u^*\varphi\, d \Div \A - \int_\Omega u \, \A\cdot \nabla\varphi\, dx. 
\end{equation}
The distribution \((\A, Du)\) is a Radon measure in \(\Omega\),
absolutely continuous with respect to \(|Du|\)
(see \cite[Theorem 1.5]{Anz} and \cite[Theorem 3.2]{ChenFrid}),
hence the equation
\begin{equation}\label{f:anz}
	\Div(u\A) = u^* \Div\A + (\A, Du)
\end{equation}
holds in the sense of measures in \(\Omega\).
(We remark that, in \cite{ChenFrid}, the measure \((\A, Du)\) is denoted
by \(\overline{\A\cdot Du}\).)
Furthermore, Chen and Frid in \cite{ChenFrid} proved that the absolutely continuous part
of this measure with respect to the Lebesgue measure is given by
\(
(\A, Du)^a = \A \cdot \nabla u\, \LLN
\).

\section{Assumptions on the vector field and preliminary results}
\label{s:div}

As we have explained in the Introduction,
we are willing to compute the divergence of the composite function
$\v(x) := \B(x,u(x))$ with $u\in BV$,
where $\B(\cdot, t)\in\DM$ and $\B(x, \cdot)\in C^1$.
Nevertheless, it will be convenient to state our assumptions on the vector field
$\bsmall(x,t) := \partial_t \B(x,t)$.

In Section \ref{ss:assumptions} we list the assumptions on $\bsmall$ and
we prove a number of basic properties of $\bsmall$ and $\B$.

Then, in Section \ref{distrtraces} we prove some regularity result of the weak normal traces of
$\B$.

Finally, in Section \ref{ss:basic}, we prove that $\v \in \DM$.

\subsection{Assumptions on  the vector field \(\B\).}
\label{ss:assumptions}
In this section we list and comment all the assumptions on the vector field
\(\B(x,t)\).

Let $\Omega\subset\R^N$ be a non-empty open set.
Let \(\bsmall\colon \Omega\times \R \to \R^N\) be a function
satisfying the following assumptions:
\begin{itemize}
\item[(i)]
\(\bsmall\) is a locally bounded Borel function;

\item[(ii)]
for \(\LLN\)--a.e.\ \(x\in\Omega\),
the function \(\bsmall(x, \cdot)\) is continuous in \(\R\);

\item[(iii)]
for every \(t\in\R\), \(\bsmall(\cdot, t)\in\DMloc\);

\item[(iv)]
the least upper bound
\[
\sigma\defeq \bigvee_{t\in \R} |\Div_x \bsmall(\cdot,t)|
\]
is a Radon measure.
(See \cite[Def.~1.68]{AFP} for the definition of least upper
bound of measures.)
\end{itemize}

We remark that, since $\Div_x \bsmall(\cdot, t) \ll \H^{N-1}$ for every $t\in\R$,
then also $\sigma \ll \H^{N-1}$.

\smallskip

From (i) and Proposition~\ref{p:ScorzaDragoni} it follows that
there exists a set \(Z_1\subset\R^N\) such that \(\LLN(Z_1) = 0\) and,
for every \(t\in\R\), the function \(x\mapsto \bsmall(x, t)\)
is approximately continuous on \(\R^N\setminus Z_1\).

By definition of least upper bound of measures, we have that
\(\Div_x\bsmall(\cdot, t) \ll \sigma\) for every \(t\in\R\).
If we denote by \(\f(\cdot, t)\) the Radon--Nikod\'ym derivative
of \(\Div_x\bsmall(\cdot, t)\) with respect to \(\sigma\), we have
\[
\Div_x\bsmall(\cdot, t) = \f(\cdot, t)\, \sigma,
\quad
\f(\cdot, t) \in L^1(\sigma), \qquad \forall t\in\R.
\]
Moreover, since \(|\Div_x \bsmall(\cdot, t)| \leq \sigma\), we have that
\begin{equation}\label{f:bbound}
\forall t\in\R:
\quad
|\f(x,t)| \leq 1 \quad
\text{for \(\sigma\)-a.e.}\ x\in\Omega.
\end{equation}

Let us extend $\bsmall$ to $0$ in $(\R^N\setminus\Omega)\times\R$,
so that the vector field
\begin{equation}\label{f:B}
\B(x,t) := \int_0^t \bsmall(x,s)\, ds,
\qquad x\in\R^N,\ t\in\R,
\end{equation}
is defined for all $(x,t)\in\R^N\times\R$.
Moreover \(\B(x,0) = 0\) for every \(x\in\R^N\)
and, from (ii), for every \(x\in\R^n\)
one has \(\bsmall(x,t) = \partial_t \B(x,t)\) for every \(t\in\R\).

\begin{lemma}\label{l:B}
For every \(t\in\R\) it holds
\(\B(\cdot, t)\in\DMloc\) and \(\Div_x\B(\cdot, t) \ll \sigma\).
If we denote by \(\F(\cdot, t)\) the Radon--Nikod\'ym derivative
of \(\Div_x\B(\cdot, t)\) with respect to \(\sigma\), we have that
\begin{equation}\label{f:F}
\F(x,t) = \int_0^t \f(x, s)\, ds,
\qquad \text{for \(\sigma\)--a.e.}\ x\in\Omega, 
\end{equation}
and, for every \(t,s\in\R\),
\begin{equation}\label{f:estiv}
|\F(x,t) - \F(x,s)| \leq |t-s|
\qquad
\text{for \(\sigma\)--a.e.}\ x\in\Omega.
\end{equation}
\end{lemma}

\begin{proof}
Let \(\varphi\in C^1_c(\Omega)\). We have that
\[
\begin{split}
\int_{\Omega} \nabla\varphi(x)\cdot \B(x,t)\, dx
& =
\int_{\Omega} \nabla\varphi(x)\cdot \int_0^t\bsmall(x,s)\,ds\, dx
\\ & =
-\int_0^t \left(\int_{\Omega} \varphi(x)\, \f(x,s)\, d\sigma(x)\right)\, ds
\\ & =
-\int_{\Omega} \varphi(x) \left(\int_0^t \, \f(x,s)\, ds\right) d\sigma(x)\,.
\end{split}
\]
Hence \(\Div_x\B(\cdot, t)\) is a Radon measure, it is absolutely continuous with respect to \(\sigma\)
and its Radon--Nikod\'ym derivative with respect to \(\sigma\)
is \(\F(\cdot, t)\).

For every non-negative \(\varphi\in C^1_c(\Omega)\) and for every \(t,s\in\R\), taking into account that for every \(s\in\R\), \(|\f(x,s)|\leq 1\)
for \(\sigma\)--a.e.\ \(x\in\Omega\), an analogous integration gives
\[
\begin{split}
\left|\int_{\Omega} \nabla\varphi(x)\cdot \left[\B(x,t) - \B(x,s)\right]\, dx\right|
& =
\left|\int_s^t \left(\int_{\Omega} \varphi(x)\, \f(x,w)\, d\sigma(x)\right)\,dw\right|
\\ & \leq
|t-s| \int_{\Omega} \varphi(x) d\sigma(x)\,,
\end{split}
\]
so that
\[
|\Div_x\B(\cdot, t)(A) - \Div_x\B(\cdot, s)(A)|
\leq |t-s|\, \sigma(A)
\]
for every Borel set \(A\Subset\Omega\), hence \eqref{f:estiv} follows.
\end{proof}

\begin{lemma}\label{l:lebF}
	There exists a set \(Z_2\subset\Omega\),
	with \(\sigma(Z_2) = 0\) and \(\H^{N-1}(Z_2) = 0\), such that
	every \(x_0\in\Omega\setminus Z_2\)
	is a Lebesgue point of \(F(\cdot, t)\) 
	with respect to the measure \(\sigma\)
	for every \(t\in\R\).	
\end{lemma}

\begin{proof}
	Let \(S\subset\R\) be a countable dense subset of \(\R\) and, for every \(s\in S\),
	let $\Omega_s$ be the set of the points $x_0\in\Omega$ such that the Radon-Nikod\'ym derivative 
	\(F(\cdot, s)\) of \(\Div_x \B(\cdot, s)\) exists in $x_0$ and $x_0$ is a Lebesgue point of the function \(F(\cdot, s)\) with respect to the measure \(\sigma\).
	Define
	\[
	Z_2:=\R^N\setminus\bigcap_{s\in S}\Omega_{s}.
	\]
	Clearly $\sigma\left( Z_2\right)  =0$ and $\mathcal H^{N-1}(Z_2)=0$. 
	We claim that every $x_0\in \Omega\setminus Z_2$ is a Lebesgue point of the function 
	\(\F(\cdot, w)\)
	for all $w\in \R$.
	
	Fix $x_{0}\in\Omega\setminus Z_2$ and $w\in\R$. 
	Since the set $S$ is dense in $\R$ we may find a
	sequence $( w_{n})$ of points in \(S\) such that
	$w_{n}\rightarrow w$. 
	Then, by \eqref{f:estiv}, 
	\[
	\begin{split}
	|\F(x_0, w) - \F(x,w)| \leq {} &
	|\F(x_0, w) - \F(x_0, w_n)|
	+ |\F(x_0, w_n) - \F(x, w_n)|  
	\\ & + 
	|\F(x,w_n) - \F(x,w)|
	\\ \leq {} &
	|\F(x_0, w_n) - \F(x, w_n)| + 
	2\,|w-w_n|.
	\end{split}
	\]
	By averaging over $B_r(x_{0})$ and letting $r\rightarrow0^{+}$
	in the previous inequality we get
	\[
	\limsup_{r\rightarrow0^{+}}\frac{1}{\sigma\left(  B_r(x_0)  \right)  }  \int_{B_r(x_0)  }
	\left|
	\F(x_0,w) - \F(x,w)
	\right|  \,d\sigma\\
	\leq
	2\, |w-w_n|,
	\]
	where we have used the fact that $x_{0}$ is a Lebesgue point of \(\F(\cdot, w_n)\)
	w.r.t.\ the measure $\sigma$. 
	Since $w_{n}\rightarrow w$ letting
	$n\rightarrow\infty$, the previous inequality proves that
	every $x_0\in \Omega\setminus Z_2$ is a Lebesgue point of the function 
	\(\F(\cdot, w)\)
	for all $w\in\R$.
\end{proof}

\subsection{Weak normal traces of $\B$}
\label{distrtraces}
Let \(\Sigma\subset\Omega\) be
an oriented countably \(\H^{N-1}\)--rectifiable set.
By Lemma~\ref{l:B}, it follows that, for almost  every \(t\in\R\),
the traces of the normal component of the vector field \(\B(\cdot, t)\)
can be defined as distributions
\(\text{Tr}^\pm(\B(\cdot, t), \Sigma)\)
in the sense of Anzellotti
(see \cite{AmbCriMan,Anz,ChenFrid}).
It turns out that these distributions are induced by \(L^\infty\) functions on \(\Sigma\),
still denoted by \(\text{Tr}^\pm(\B(\cdot, t), \Sigma)\), and
\[
\|\text{Tr}^\pm(\B(\cdot, t), \Sigma)\|_{L^\infty(\Sigma, \H^{N-1})}
\leq \|\B(\cdot, t)\|_{L^\infty(\Omega)}.
\]

More precisely, 
let us briefly recall the construction given in \cite{AmbCriMan} (see Proposition~3.4 and Remark~3.3 therein).
Since \(\Sigma\) is countably $\mathcal{H}^{N-1}$--rectifiable,
we can find countably many \textsl{oriented} \(C^1\) hypersurfaces \(\Sigma_i\),
with classical normal \(\nu_{\Sigma_i}\),
and pairwise disjoint Borel sets \(N_i\subseteq \Sigma_i\)
such that \(\mathcal{H}^{N-1}(\Sigma\setminus \bigcup_i N_i) = 0\).

Moreover, it is not restrictive to assume that, for every \(i\),  
there exist two open bounded sets \(\Omega_i, \Omega'_i\) with \(C^1\) boundary
and outer normal vectors \(\nu_{\Omega_i}\) and \(\nu_{\Omega_i'}\) respectively,
such that
\(N_i\subseteq \partial\Omega_i \cap \partial\Omega'_i\)
and
\[
\nu_{\Sigma_i}(x) = \nu_{\Omega_i}(x) = -\nu_{\Omega'_i}(x)
\qquad \forall x\in N_i.
\]
At this point we choose, on \(\Sigma\), the orientation given by
\(\nu_{\Sigma}(x) := \nu_{\Sigma_i}(x)\)
\(\mathcal{H}^{N-1}\)-a.e.\ on \(N_i\).

Using the localization property proved in \cite[Proposition 3.2]{AmbCriMan},
for every \(t\in\R\) we can define
the normal traces of \(\B(\cdot, t)\) on \(\Sigma\) by
\begin{equation}\label{f:beta}
\beta^-(\cdot, t) := \text{Tr}(\B(\cdot, t), \partial\Omega_i),
\quad
\beta^+(\cdot, t) := -\text{Tr}(\B(\cdot, t), \partial\Omega'_i),
\qquad
\mathcal{H}^{N-1}-\text{a.e.\ on}\ N_i.
\end{equation}

These two normal traces belong to
\(L^{\infty}(\Sigma)\) (see \cite[Proposition 3.2]{AmbCriMan})
and
\begin{equation}\label{mmm}
\Div_x \B(\cdot, t) \res\Sigma =
\left[\beta^{+}(\cdot,t)-\beta^{-}(\cdot,t)\right]\, {\mathcal H}^{N-1} \res\Sigma\,.
\end{equation}

\begin{lemma}\label{l:lip}
The maps \(t\mapsto \beta^\pm(\cdot, t)\) are Lipschitz continuous from 
\(\R\) to \(L^\infty(\Sigma)\).
More precisely
\begin{equation}\label{f:betalip}
\|\beta^\pm(\cdot,t) - \beta^\pm(\cdot,s)\|_{L^\infty(\Sigma)} \leq 
\|\bsmall\|_{\infty}\, |t-s|
\qquad\forall t,s\in\R.
\end{equation}	
\end{lemma}

\begin{proof}
It is enough to observe that, for every \(i\),
the map \(\boldsymbol{X} \mapsto \text{Tr}(\boldsymbol{X}, \partial\Omega_i)\)
is linear in \(\DM\) and, by Proposition 3.2 in \cite{AmbCriMan}, it holds
\begin{equation}
\label{bella}
\|\text{Tr}(\boldsymbol{X}, \partial\Omega_i)\|_{L^\infty(\partial\Omega_i)}
\leq \|\boldsymbol{X}\|_{L^\infty(\Omega_i)}
\qquad\forall \boldsymbol{X}\in\DM,
\end{equation}
and the same inequality holds when \(\Omega_i\) is replaced by \(\Omega_i'\).
Hence \eqref{f:betalip} follows from Assumption~(i).
\end{proof}

In Step 6 of the proof of Theorem~\ref{chainb4}
we will use these normal traces on the set
\(\Sigma := J_u\),
where \(J_u\) is the jump set of a \(BV\)
function \(u\).
More precisely,
we will use a slightly stronger property,
stated in the following proposition.

\begin{proposition}
\label{p:lip}
There exist a set \(\Sigma'\subseteq\Sigma\),
with \(\H^{N-1}(\Sigma\setminus\Sigma') = 0\),
and representatives \(\overline{\beta}^\pm\) of
\(\beta^\pm\)
(in the sense that, for every \(t\in\R\), \(\overline{\beta}^\pm(\cdot, t)
=\beta^\pm(\cdot, t)\) \(\H^{N-1}\)--a.e.\ in \(\Sigma\))
such that
\begin{equation}\label{f:lip}
|\overline{\beta}^\pm(x,t)
- \overline{\beta}^\pm(x,s)|
\leq \|\bsmall\|_\infty\, |t-s|
\qquad
\forall x\in\Sigma',\ t,s\in\R.
\end{equation}
\end{proposition}

\begin{proof}
Let \(Q\subset\R\) be a countable dense subset of \(\R\).
There exists a set \(\Sigma'\subseteq\Sigma\),
with \(\H^{N-1}(\Sigma\setminus\Sigma') = 0\),
such that
\begin{equation}\label{f:lipQ}
|\beta^\pm(x,p)
- \beta^\pm(x,q)|
\leq M|p-q|
\qquad
\forall x\in\Sigma',\ p,q\in Q,
\end{equation}
where \(M := \|\bsmall\|_\infty\).
Let us define the functions
\(\overline{\beta}^\pm\) in the following way.
If \(q\in Q\), we define
\(\overline{\beta}^\pm(x,q) = \beta^\pm(x,q)\)
for every \(x\in\Sigma\).
If \(t\in\R\setminus Q\), let \((q_j)\subset Q\)
be a sequence converging to \(t\), and define
\[
\overline{\beta}^\pm(x,t) :=
\begin{cases}
\lim\limits_{j\to+\infty} \beta^\pm(x, q_j),
&\text{if}\ x\in\Sigma',\\
\beta^\pm(x,t), &
\text{if}\ x\in \Sigma\setminus\Sigma'.
\end{cases}
\]
(We remark that, for every \(x\in\Sigma'\),
by \eqref{f:lipQ}
\((\beta^\pm(x,q_j))_j\) is a Cauchy sequence, hence it is convergent, in \(\R\).
Moreover, its limit is independent of the choice of the sequence
\((q_j)\subset Q\) converging to \(t\).) 
From Lemma~\ref{l:lip} we have that
\[
|\beta^\pm(x,q_j) - \beta^\pm(x,t)|
\leq M|q_j - t|
\qquad
\text{for \(\H^{N-1}\)--a.e.}\ x\in\Sigma.
\]
Passing to the limit as \(j\to +\infty\), it follows that
\(\overline{\beta}^\pm(\cdot , t) = \beta^\pm(\cdot, t)\) 
\(\H^{N-1}\)--a.e.\ on \(\Sigma\).

Let \(t,s\in\R\), and let \((t_j), (s_j) \subset Q\) be two sequences in \(Q\)
converging respectively to \(t\) and \(s\).
From \eqref{f:lipQ} we have that
\[
|\beta^\pm(x,t_j)
- \beta^\pm(x,s_j)|
\leq M|t_j-s_j|
\qquad
\forall x\in\Sigma',\ j\in\N,
\]
hence \eqref{f:lip} follows
passing to the limit as \(j\to +\infty\).
\end{proof}

\begin{remark}
In what follows we will always denote by \(\beta^\pm\)
the representatives \(\overline{\beta}^\pm\)
of Proposition~\ref{p:lip}.
\end{remark}

\subsection{Basic estimates on the composite function}
\label{ss:basic}

In this section we shall use a regularization argument to prove that 
the composite function $\v(x) := \B(x,u(x))$ belongs to $\DMloc$.

\begin{lemma}\label{l:Beps}
Let $\bsmall\colon\Omega\times\R\rightarrow\R^N$ 
satisfy assumptions (i), (ii), 
extended to $0$ in $(\R^N\setminus\Omega)\times\R$,
and let \(\B\colon\R^N\times\R\to\R^N\) be defined by \eqref{f:B}.
For every \(\varepsilon > 0\) and every \(t\in\R\) let
\(\bsmall_\varepsilon(\cdot, t) := \rho_{\varepsilon} \ast \bsmall(\cdot, t)\)
and
\(\B_\varepsilon(\cdot, t) := \rho_{\varepsilon} \ast \B(\cdot, t)\).
Then it holds:
\begin{itemize}
\item[(a)]
 there exists an \(\LLN\)--null set \(Z\subset\Omega\) such that
\[
\lim_{ \varepsilon\to 0^+} \bsmall_\varepsilon(x,t) =\bsmall(x,t)\,,
\quad
\lim_{ \varepsilon\to 0^+} \B_ \varepsilon(x,t) =\B(x,t)\,,
\qquad \forall (x,t)\in (\Omega\setminus Z) \times \R;
\]
\item[(b)]
\(\partial_t \B_ \varepsilon(x,t)=\bsmall_\varepsilon(x,t)\) 
for every \((x,t)\in\Omega_\varepsilon\times\R\),
where $\Omega_\varepsilon := \{x\in \Omega:\ \text{dist}(x, \partial\Omega) > \varepsilon\}$.
\end{itemize}
\end{lemma}

\begin{proof}
The conclusion (a) for \(\bsmall_\varepsilon\)
follows directly from Corollary~\ref{mmmmm}.

By the Fubini--Tonelli theorem we have that
\begin{equation}\label{f:derB}
\B_\varepsilon(x,t) = \int_0^t \bsmall_\varepsilon(x, s)\, ds,
\qquad
\forall (x,t) \in \Omega_\varepsilon \times \R,
\end{equation}
hence (b) follows.
Moreover, for every \(x\in\R^N\setminus Z\) and every \(t\in\R\),
the conclusion (a) for \(\B_\varepsilon\) follows from the
Dominated Convergence Theorem.
\end{proof}

\begin{lemma}\label{l:divv}
Let $\bsmall$
satisfy assumptions (i)--(iv)
and let \(\B\) be defined by \eqref{f:B}.
Then, for every $u\in \BVLloc$,
the function $\v\colon\Omega\rightarrow\R^N$,
defined by
\[
\v(x):=\B(x,u(x))
\,,
\qquad x\in\Omega,
\]
belongs to $\DMloc$ and
\begin{equation}\label{f:boundv}
|\Div\v|(K) \leq 
\|u\|_{L^\infty(K)} \sigma(K) +
\|\bsmall\|_\infty |Du|(K)
\qquad \forall K\subset\Omega,\ K\ \text{compact}.
\end{equation}
Moreover, the functions \(\v_\epsilon (x) := \B_\varepsilon(x, u(x))\)
converge to \(\v\) 
a.e.\ in \(\Omega\) and
in \(L^1_{\rm loc}(\Omega)\).
\end{lemma}

\begin{remark}
In the following we shall use the notation
\[
\Div_x \B(x, u(x)) :=
(\Div_x \B) (x, u(x)) =
\left.\Div_x \B (x, t)\right|_{t = u(x)}. 
\]
The divergence of the composite function $x\mapsto \B(x,u(x))$ will be denoted by
$\Div [\B(x, u(x))]$.
\end{remark}

\begin{proof}
Using the same notation of Lemma~\ref{l:Beps}, one has
\[
\Div_x \B_\varepsilon(\cdot, t) = F(\cdot, t)\, \rho_\varepsilon \ast \sigma.
\]
Since \(F(x,t) = \int_0^t f(x,s)\, ds\) and \(|f(x,s)|\leq 1\),
we have that, for every \(t\in\R\),
\begin{equation}\label{f:stdivx}
|\Div_x \B_\varepsilon(x, t)|
\leq |t|\, \rho_\varepsilon \ast \sigma(x),
\qquad
\text{for \(\sigma\)-a.e.}\ x\in\Omega_\varepsilon.
\end{equation}

Let us define \(\v_\varepsilon(x) := \B_\varepsilon(x, u(x))\).
Consider first the case \(u\in C^1(\Omega) \cap L^\infty(\Omega)\).
We have that
\[
\Div \v_\varepsilon(x) = \Div_x \B_\varepsilon(x, u(x)) + 
\pscal{\bsmall_\varepsilon(x,u(x))}{\nabla u(x)}\,,
\]
hence, from \eqref{f:stdivx}, for every compact subset \(K\) of \(\Omega\)
and for every $\varepsilon > 0$ small enough
it holds
\[
\int_K |\Div \v_\varepsilon|\, dx
\leq
\|u\|_{L^\infty(K)} 
\int_K
\rho_{\varepsilon}\ast \sigma (x)\, dx
+\|\bsmall\|_\infty \int_K |\nabla u(x)|\, dx.
\]
By an approximation argument (see \cite[Theorem 5.3.3]{Zi}), when 
\(u\in\BVLloc\)
we get
\[
\int_K |\Div \v_\varepsilon|\, dx
\leq
\|u\|_{L^\infty(K)} 
\int_K
\rho_{\varepsilon}\ast \sigma (x)\, dx
+\|\bsmall\|_\infty |D u|(K).
\]
Finally, \eqref{f:boundv} follows observing that,
by Lemma~\ref{l:Beps}(a),
\(\B_\varepsilon(x, u(x)) \to \B(x, u(x))\) for \(\LLN\)-a.e.\ \(x\in\Omega\),
hence \(\v_\varepsilon \to \v\) in \(L^1_{\rm loc}(\Omega)\).
\end{proof}

\section{Main results}
\label{ss:div2}

The main results of the paper are stated in
Theorems~\ref{chainb4} and~\ref{chainb5}.
As a preliminary step, we will prove Theorems~\ref{chainb4} under
the additional regularity assumption \(u\in W^{1,1}\).

\begin{theorem}[$\mathcal{DM}^{\infty}$-dependence and $u\in W^{1,1}$]
\label{chainb3}
Let $\bsmall$
satisfy assumptions (i)--(iv)
and let \(\B\) be defined by \eqref{f:B}.
Then, for every $u\in W^{1,1}_{\rm loc}(\Omega)\cap L^{\infty}_{\rm loc}(\Omega)$,
the function $\v\colon\Omega\rightarrow\R^N$,
defined by
\begin{equation}\label{f:defv}
\v(x):=\B(x,u(x))
\,,
\qquad x\in\Omega,
\end{equation}
belongs to $\DMloc$ and 
the following equality holds in the sense of measures:
\begin{equation}\label{f:chainw}
\Div\v  = F(x, \ut(x))\, \sigma + \pscal{\partial_t\B(x, u(x))}{\nabla u(x)}\mathcal L^N,
\end{equation}
where \(F(\cdot, t)\) is the Radon--Nikod\'ym derivative of \(\Div_x\B(\cdot,t)\)
with respect to \(\sigma\)
(see Section~\ref{ss:assumptions}).
\end{theorem}

\begin{remark}\label{r:not}
With some abuse of notation, we will also write equation \eqref{f:chainw}
as
\begin{equation}\label{lolo33}
\Div\,\v=\left.\Div_x\B(x,t)\right|_{t=\ut(x)}+\pscal{\partial_t\B(x, u(x))}{\nabla u(x)}\mathcal L^N\,.
\end{equation}
\end{remark}

\begin{proof}
By Lemma~\ref{l:divv}
the function $\v$ belongs to $\DMloc$ and 
satisfies \eqref{f:boundv}.

Since all results are of local nature in the space variables,
it is not restrictive to assume that
$\Omega = \R^N$, $\bsmall$ is a bounded Borel function,
and $\bsmall(\cdot, t) \in \DM$ for every $t\in\R$.

We will use a regularization argument
as in \cite[Theorem~3.4]{DCFV2}.
More precisely, 
as in Lemma~\ref{l:divv} let
\(\B_ \varepsilon(\cdot, t) := \rho_ \varepsilon \ast \B(\cdot, t)\)
and \(\v_\varepsilon(x) :=\B_ \varepsilon(x,u(x))\).
Since $\B_\varepsilon$ is a Lipschitz function in $(x,t)$, by using the chain rule formula of Ambrosio and Dal Maso
(see \cite[Theorem 3.101]{AFP}) one has
\begin{equation}\label{alpha1}
\begin{split}
\int_{\R^{N}} \pscal{\nabla\phi(x)}{\v_\varepsilon(x)}\, dx = {} &
-\int_{\R^{N}} \phi(x) \Div_x \B_ \varepsilon (x,u(x))\,dx
\\ & -
\int_{\R^{N}} \phi(x)\, \pscal{\bsmall_\varepsilon(x, u(x))}{\nabla u(x)}\, dx,
\end{split}
\end{equation}
and the claim will follow by passing to the limit as \( \varepsilon\to 0^+\).  

Namely, by Lemma \ref{l:divv},
\(\v_\varepsilon(x) \to \v(x)\) for $\mathcal L^N$--a.e.\ $x\in\R^N$.
Then by Dominated Convergence Theorem we have
\begin{equation}\label{alpha2}
\lim_{ \varepsilon\to 0^+} \int_{\R^{N}} \pscal{\nabla\phi(x)}{\v_ \varepsilon(x)}\, dx =\int_{\R^{N}} \pscal{\nabla\phi(x)}{\v(x)}\, dx \,.
\end{equation}
This is equivalent to say
\begin{equation}\label{sinistra1}
\Div [\B_\varepsilon(x,u(x))]\,\mathcal L^N \overset{*}{\rightharpoonup} 
\Div\v (x)
\,,\quad     \hbox{ as } \varepsilon\to 0^+
\end{equation}
in the weak${}^*$ sense of measures.

Similarly, from Lemma~\ref{l:Beps}(a) we have that,
for $\mathcal L^N$--a.e.\ $x\in\R^N$,
\[
\lim_{ \varepsilon\to 0^+} \bsmall_\varepsilon(x,u(x)) = \bsmall(x,u(x))\,.
\]
Then by (ii) and the Dominated Convergence Theorem we have
\begin{equation}\label{alpha3}
\lim_{ \varepsilon\to 0^+} \int_{\R^{N}} 
\phi(x)\,
\pscal{\bsmall_ \varepsilon(x, u(x))}{\nabla u(x)}\, dx 
=\int_{\R^{N}} 
\phi(x)\,
\pscal{\bsmall(x, u(x))}{\nabla u(x)}\, dx \,.
\end{equation}

It remains to prove that
\begin{equation}\label{f:ie}
I_\varepsilon := 
\int_{\R^{N}} \phi(x) \Div_x \B_\varepsilon (x,u(x))\,dx
\ \xrightarrow{\varepsilon\to 0}\
\int_{\R^N} \phi(x) F(x, \ut(x))\, d\sigma(x).
\end{equation}
Assume first that \(u\geq 0\) and let \(C > \|u\|_\infty\).
Let us rewrite $I_\varepsilon$ in the following way: 
\[
\begin{split}
I_\varepsilon & =
\int_{\R^N} \phi(x) \int_0^{u(x)} \Div_x \bsmall_\varepsilon(x,t)\, dt \, dx
\\ & =
\int_0^C dt \int_{\R^N} \phi(x) \chi_{\{u > t\}}(x)
\int_{\R^N} \rho_\varepsilon(x-y)\, d\Div_y\bsmall(y,t)
\\ & =
\int_0^C dt \int_{\R^N} \rho_\varepsilon\ast (\phi \chi_{\{u > t\}})(y)
\, d\Div_y\bsmall(y,t)\,.
\end{split}
\]
By Corollary 3.80 in \cite{AFP} it holds
\[
\rho_\varepsilon\ast (\phi \chi_{\{u > t\}})(x) \to
\phi(x) \chi^*_{\{u > t\}}(x)
= \phi(x) \chi_{\{\ut > t\}}(x)
\qquad
\text{for\ \(\H^{N-1}\)-a.e.}\ x
\]
(hence for \(\Div_x \bsmall(\cdot, t)\)-a.e.\ \(x\)).
Passing to the limit as \(\varepsilon\to 0\) we get
\[
\begin{split}
\lim_{\varepsilon\to 0} I_\varepsilon & =
\int_0^C dt \int_{\R^N} \phi(x) \chi_{\{\ut > t\}}(x)
\, d\Div_x\bsmall(x,t)
\\ & =
\int_{\R^N} \phi(x) \int_0^{\ut(x)}
\, f(x,t)\, dt \, d\sigma(x)
\\ & =
\int_{\R^N} \phi(x) F(x, \ut(x))\, d\sigma(x),
\end{split}
\]
hence \eqref{f:ie} is proved in the case $u\geq 0$.

The general case can be handled similarly.
Namely, the integral \(I_\varepsilon\) can be written as
\[
\begin{split}
I_\varepsilon & =
\int_{-C}^C dt \int_{\R^N} \phi(x) \chi_{{u,t}}(x)
\int_{\R^N} \rho_\varepsilon(x-y)\, d\Div_y\bsmall(y,t)
\\ & =
\int_{-C}^C dt \int_{\R^N} \rho_\varepsilon\ast (\phi \chi_{{u,t}})(y)
\, d\Div_y\bsmall(y,t)\,,
\end{split}
\]
where \(\chi_{u,t}\) is the characteristic function of the set
\[
\{x\in\R^N:\ t\ \text{belongs to the segment of endpoints $0$ and $u(x)$}\},
\]
and the limit as \(\varepsilon\to 0\) can be computed
exactly as in the previous case.
\end{proof}

\bigskip
We now state the main results of the paper; the proofs are collected at the end of the section.

\begin{theorem}[$\mathcal{DM}^{\infty}$-dependence and $u\in BV$]
\label{chainb4}
Let $\bsmall$
satisfy assumptions (i)--(iv),
let \(\B\) be defined by \eqref{f:B},
and let $u\in \BVLloc$.
Then the distribution $(\bsmall(\cdot,u), Du)$,
defined by
\begin{equation}
\label{f:bxu}
\begin{split}
 \langle(\bsmall(\cdot,u), Du),\varphi\rangle  := {} &
 -\frac{1}{2}\int_\Omega
\left[ F(x, u^+(x)) + F(x, u^-(x)) \right]\, \varphi(x)\, d\sigma(x)
\\ & -\int_\Omega \B(x,u(x))\cdot\nabla
    \varphi(x)\, dx,
\qquad \forall\varphi\in C_c^\infty(\Omega),
\end{split}
\end{equation}
is a Radon measure in $\Omega$, and satisfies
\begin{equation}\label{f:mubdd}
|(\bsmall(\cdot,u), Du)|(E)\leq
\|\bsmall\|_\infty
|Du|(E),
\qquad
\text{for every Borel set}\
E\subset\Omega.
\end{equation}
In other words,
the composite function $\v\colon\Omega\rightarrow\R^N$,
defined by
$\v(x):=\B(x,u(x))$,
belongs to $\DMloc$,
and the following equality holds in the sense of measures:
\begin{equation}\label{lolo}
\begin{split}
\Div\,\v =
\frac{1}{2}\left[ F(x, u^+(x)) + F(x, u^-(x)) \right]\, \sigma
+(\bsmall(\cdot,u), Du).
\end{split}
\end{equation}
\end{theorem}

\begin{remark}
The measure $(\bsmall(\cdot,u), Du)$
extends the notion of pairing defined by Anzellotti \cite{Anz},
in the case $\bsmall(x, w) = \A(x)$,
with $\A\in\DM$.
\end{remark}

\begin{proposition}[Traces of the composite function]\label{p:traces}
Let the assumptions of Theorem~\ref{chainb4} hold,
and let $\Sigma\subset\Omega$ be a countably $\H^{N-1}$-rectifiable set,
oriented as in Section~\ref{distrtraces}.
Then the normal traces on $\Sigma$ of the composite function $\v\in\DMloc$, defined at \eqref{f:defv},
are given by
\begin{equation}
\label{f:tracesv}
\Trace{\v}{\Sigma} =
\begin{cases}
\beta^\pm(x, u^\pm(x)),
&\text{for $\H^{N-1}$-a.e.}\ x\in J_u,\\
\beta^\pm(x, \ut(x)),
&\text{for $\H^{N-1}$-a.e.}\ x\in \Sigma\setminus J_u,
\end{cases}
\end{equation}
where, for every $t\in\R$, $\beta^\pm(\cdot, t)$ are the normal traces of $\B(\cdot, t)$
on $\Sigma$ (see \eqref{mmm}).
With our convention $u^\pm(x) = \ut(x)$ if $x\in\Omega\setminus S_u$, \eqref{f:tracesv} can
be written as
$\Trace{\v}{\Sigma} = \beta^\pm(x, u^\pm(x))$ for $\H^{N-1}$-a.e.\ $x\in\Sigma$.
\end{proposition}

\begin{theorem}[Representation of the pairing measure]\label{chainb5}
Let the assumptions of Theorem~\ref{chainb4} hold,
and consider the standard decomposition of 
the measure
$\mu := (\bsmall(\cdot,u), Du)$ as
\[
\mu=\mu^{ac}+\mu^c+\mu^j,\qquad 
\mu^{d}:=\mu^{ac}+\mu^{c}.
\]
Then
\begin{gather*}
\mu^{ac}=\pscal{\bsmall(x, \ut(x))}{\nabla u(x)}\mathcal L^N,
\\
\mu^j=
[\beta^*(x,u^+(x))
-\beta^*(x,u^-(x))]
\hh\res {J_u},
\end{gather*}
where, for every $t\in\R$, $\beta^\pm(\cdot, t)$ are the normal traces of $\B(\cdot, t)$
on $J_u$ and $\beta^*(\cdot, t) := [\beta^+(\cdot, t) + \beta^-(\cdot, t)]/2$.

Moreover, if 
there exists a countable dense set $Q \subset\R$ such that
\begin{equation}\label{f:assumpDc}
|D^c u| (S_{\bsmall(\cdot, t)}) = 0
\qquad
\forall t\in Q,
\end{equation}
then
\[
\mu^{d}=\pscal{\widetilde{\bsmall}(x, \ut(x))}{D^d u}.
\]
Therefore, under this additional assumption
the following equality holds in the sense of measures:
\begin{equation}\label{lolo1}
\begin{split}
\Div\,\v = {} &
\frac{1}{2}\left[ F(x, u^+(x)) + F(x, u^-(x)) \right]\, \sigma
\\
& +\pscal{\widetilde{\bsmall}(x, \ut(x))}{D^d u}
+[\beta^*(x,u^+)
-\beta^*(x,u^-)]
\hh\res {J_u}.
\end{split}
\end{equation}
\end{theorem}

\begin{remark}\label{r:ipoCantor}
	Since \(\LLN({S_{\bsmall(\cdot, t)}}) = 0\) for every $t\in\R$,
	assumption \eqref{f:assumpDc} 
	is equivalent to \(|D^d u|(S_{\bsmall(\cdot, t)}) = 0\)
	for every $t\in Q$.
	In particular, it is satisfied, for example,
	if \({S_{\bsmall(\cdot, t)}}\) is 
	$\sigma$--finite with respect to $\H^{N-1}$, for every $t\in Q$
	(see \cite[Proposition~3.92(c)]{AFP}).
	This is always the case if 
	\(\bsmall(\cdot, t)\in BV_{{\rm loc}}(\Omega, \R^N)\cap L^\infty_{{\rm loc}}(\Omega, \R^N)\).
	Another relevant situation for which \eqref{f:assumpDc} holds
	happens when $D^c u = 0$,  
	i.e.\ if 
	\(u\) is a special function of bounded variation,
	e.g.\  
	if \(u\) is
	the characteristic function of a set of finite perimeter.
\end{remark}

\begin{remark}
For \(u\in BV_{{\rm loc}}(\Omega)\) we introduce the following notation:
\[
\begin{split}
\left.\Div_x \B(\cdot,t)\right|_{t=u(x)}
:= {} &
\frac12\left[\Div_x \B(\cdot,u^+(x))+\Div_x \B(\cdot,u^-(x))\right]
\\ := {} & 
\frac{1}{2}\left[ F(x, u^+(x)) + F(x, u^-(x)) \right]\, \sigma
\end{split}
\]
(see also Remark \ref{r:not}).
Then, with some abuse of notation, equation \eqref{lolo1}
can be written as
\begin{equation}\label{lolo1b}
\begin{split}
\Div\,\v=&
\left.\Div_x\B(x,t)\right|_{t=u(x)}\\
+&\pscal{\widetilde{\bsmall}(x, \ut(x))}{D^d u}
+[\beta^*(x,u^+)
-\beta^*(x,u^-)]
\hh\res {J_u}.
\end{split}
\end{equation}
\end{remark}

\begin{remark}[Anzellotti's pairing]
In the special case $\B(x,t) = t\, \A(x)$, with
$\A \in \DMloc$,
we have that
\[
\bsmall(x,t) = \A(x),
\quad
\sigma = |\Div\A|,
\quad
f(x,t) =\frac{d \Div\A}{d |\Div\A|}\,,
\quad
F(x,t) = t\, \frac{d \Div\A}{d |\Div\A|}\,,
\]
and formula \eqref{lolo} becomes
\[
\Div (u \A) = u^* \Div \A + (\A, Du),
\]
where $(\A, Du)$ is the Anzellotti's pairing.
\end{remark}

\bigskip

The remaining of this section is devoted to the proofs
of Theorem~\ref{chainb4}, 
Proposition~\ref{p:traces} and Theorem~\ref{chainb5}.

Since all results are of local nature in the space variables,
it is not restrictive to assume that
$\Omega = \R^N$, $\bsmall$ is a bounded Borel function,
and $\bsmall(\cdot, t) \in \DM$ for every $t\in\R$.

\begin{proof}[Proof of Theorem \ref{chainb4}]
By Lemma~\ref{l:divv}
the function $\v$ belongs to $\DMloc$ and 
satisfies~\eqref{f:boundv}.

We will use another regularization argument as in \cite{ChenFrid}.
More precisely, let \(u_ \varepsilon := \rho_ \varepsilon \ast u\)
be the standard regularization of \(u\),
and \(\v_ \varepsilon(x) := \B(x,u_ \varepsilon(x))\). 
Then, by 
Theorem~\ref{chainb3}, for any
$\phi\in C_0^1(\R^{N})$ we get
\begin{equation}\label{beta1}
\begin{split}
\int_{\R^{N}} \pscal{\nabla\phi(x)}{\v_ \varepsilon(x)}\, dx = {} &
-\int_{\R^{N}} \phi(x) \, \F(x, u_\varepsilon(x))\,   d\sigma(x)
\\ & -
\int_{\R^{N}} \phi(x)\, \pscal{\bsmall(x, u_ \varepsilon(x))}{\nabla u_ \varepsilon(x)}\, dx\,.
\end{split}
\end{equation}
Now we will pass to the limit as \( \varepsilon\to 0^+\) in each term.

STEP 1. Firstly, we note that
\begin{equation}\label{kiki}
\lim_{\varepsilon\to 0^+}\int_{\R^{N}}  \pscal{\nabla\phi(x)}{\v_ \varepsilon(x)}\, dx=
\int_{\R^{N}} \pscal{\nabla\phi(x)}{\v(x)}\, dx.
\end{equation}
Indeed, $u_ \varepsilon(x)\to u(x)$, as \( \varepsilon\to 0^+\), for a.e.\ $x$, $\B(x,\cdot)$ is Lipschitz continuous with Lipschitz constant independent of $x$ and $\B$ is locally bounded. 
Thus \eqref{kiki} holds by the Dominated Convergence Theorem.

STEP 2. We will prove that
\begin{equation}\label{kikinu}
\lim_{\varepsilon\to 0^+}\int_{\R^{N}} 
\phi(x) \, \F(x, u_ \varepsilon(x))\, d\sigma(x) =
\int_{\R^{N}} \phi(x)
\frac{1}{2}\left[
\F(x, u^+(x)) + \F(x, u^-(x))
\right]\, d\sigma(x).
\end{equation}
From \eqref{f:F} it holds
\begin{equation}\label{f:eq1}
\begin{split}
\int_{\R^{N}} \phi(x) \, \F(x,u_\varepsilon(x))\, d\sigma(x)
& = 
\int_{\R^{N}} \phi(x) \, \left[\int_{0}^{u_ \varepsilon(x)}
\f(x,w)\,dw\right]\,d\sigma(x)\,.
\end{split}
\end{equation}
Since  $u_ \varepsilon(x)\to u^*(x)$ for $\mathcal H^{N-1}$-a.e.\ $x$, 
and so also for $\sigma$-a.e.\ $x$ 
(since $\sigma\ll\mathcal H^{N-1}$), 
passing to the limit in \eqref{f:eq1} we obtain
\[
\lim_{\varepsilon\to 0^+}\int_{\R^{N}} 
\phi(x) \, \F(x, u_ \varepsilon(x))\, d\sigma(x) =
\int_{\R^{N}} \phi(x) \, \left[\int_{0}^{u^*(x)}
\f(x,w)\,dw\right]\,d\sigma(x) =: I\,.
\]
In the remaining part of the proof,
for the sake of simplicity we assume $u\geq 0$. 
We remark that the general case can be handled as
it has been illustrated at the end of the
proof of Theorem~\ref{chainb3}.

Let $C > \|u\|_{L^\infty(K)}$, where \(K\) is the support of \(\phi\). 
The integral \(I\) can be rewritten as
\[
I = 
\int_{0}^{C}\, \left[\int_{\R^{N}} \phi(x)
\chi_{\{u^* > w\}}(x)\f(x,w)\,d\sigma(x)\right]\,dw\,.
\]
On the other hand, for \(\mathcal{L}^1\)-a.e.\ \(w\in\R\) we have that
\[
\chi_{\{u^* > w\}} = 
\frac{1}{2}\left[\chi_{\{u^+ > w\}} + \chi_{\{u^- > w\}}\right],
\qquad \H^{N-1}\text{-a.e.\ (hence \(\sigma\)-a.e.) in}\ \R^N
\]
(see \cite[Lemma 2.2]{DCFV2}).
Hence we get
\[
\begin{split}
I & = 
\int_{0}^{C}\, \left[\int_{\R^{N}} \phi(x)
\frac12[\chi_{\{u^+>w\}}(x)+\chi_{\{u^->w\}}(x)]
\f(x,w)\,d\sigma(x)\right]\,dw
\\ & 
= 
\int_{\R^{N}} \phi(x)
\left[
\int_0^C \frac12[\chi_{\{u^+>w\}}(x)+\chi_{\{u^->w\}}(x)]
\f(x,w)\, dw
\right]\, d\sigma(x)
\\ &
=
\int_{\R^{N}} \phi(x)
\frac{1}{2}\left[
\F(x, u^+(x)) + \F(x, u^-(x))
\right]\, d\sigma(x)\,,
\end{split}
\]
so that \eqref{kikinu} is proved.

\smallskip

STEP 3. 
We claim that
the distribution $(\bsmall(\cdot,u), Du)$ defined at \eqref{f:bxu}
is a Radon measure,
satisfying \eqref{f:mubdd}
(and hence absolutely continuous with respect to $|Du|$).

For simplicity, let us denote by $\mu$ 
the distribution $(\bsmall(\cdot,u), Du)$ defined at \eqref{f:bxu}.
Since
\[
\mu = \Div\v-
\frac{1}{2}\left[
\F(x, u^+(x)) + \F(x, u^-(x))
\right]\, \sigma,
\]
by Lemma~\ref{l:divv} it is clear that
$\mu$ is a Radon measure and \eqref{lolo} holds. 
Moreover, by \eqref{beta1}, \eqref{kiki} and \eqref{kikinu} we have that,
for every $\phi\in C_c(\R^N)$,
\begin{equation}
\label{f:defmu}
\left\langle \mu,\phi \right\rangle=\lim_{\varepsilon\to 0^+}
\int_{\R^{N}} \phi(x)\, \pscal{\bsmall(x, u_ \varepsilon(x))}{\nabla u_ \varepsilon(x)}\, dx.
\end{equation}
Let us prove that \eqref{f:mubdd} holds.
Namely, let $U\subset\R^N$ be an open set, let $K\Subset U$
be a compact set,
and let $\phi\in C_c(\R^N)$ be a function
with support contained in $K$.
There exists $r_0>0$ such that
$K_r := K + B_r(0)\subset U$ for every $r\in (0,r_0)$.
Let $r\in(0, r_0)$ be such that $|Du|(\partial K_r) = 0$
(this property holds for almost every $r$).
Then
\[
|\pscal{\mu}{\phi}|
\leq \|\phi\|_\infty \|\bsmall\|_\infty
\liminf_{\varepsilon\to 0}\int_{K_r} |\nabla u_\varepsilon|\, dx
= \|\phi\|_\infty \|\bsmall\|_\infty
|Du| (K_r)
\leq \|\phi\|_\infty \|\bsmall\|_\infty
|Du| (U),
\]
hence
\[
|\mu|(K) \leq \|\bsmall\|_\infty
|Du| (U),
\]
so that \eqref{f:mubdd} follows
by the regularity of the Radon measures $|\mu|$ and $|Du|$.
\end{proof}

\bigskip

\begin{proof}[Proof of Proposition~\ref{p:traces}]
We will use the same notations of Section~\ref{distrtraces}.
It is not restrictive to assume that $J_u$ is oriented with $\nu_\Sigma$
on $J_u \cap \Sigma$.

Since, by Theorem~\ref{chainb4}, \(\v\in\DM\),
there exist the weak normal traces 
of \(\v\) on \(\Sigma\).
Let us prove \eqref{f:tracesv} for $\Tr^-$.

Let \(x\in \Sigma\) satisfy:
\begin{itemize}
	\item[(a)] 
	\(x\in (\R^N\setminus S_u)\cup J_u\),
	\(x\in N_i\) for some \(i\), the set \(N_i\) has density \(1\) at \(x\) 
	and \(x\) is a Lebesgue point of \(\beta^-(\cdot, t)\),
	with respect to
	\(\mathcal{H}^{N-1}\res\partial\Omega_i\),
	for every \(t\in\R\);
	
	\item[(b)]
	\(\sigma\res\Omega_i (B_\varepsilon(x)) = o(\varepsilon^{N-1})\) as \(\varepsilon\to 0\);
	
	\item[(c)]
	\(|\Div\v|\res\Omega_i  (B_\varepsilon(x)) = o(\varepsilon^{N-1})\).
\end{itemize}
We remark that \(\H^{N-1}\)-a.e.\ \(x\in \Sigma\) satisfies these conditions.
In particular, (a) is satisfied thanks to Proposition~\ref{p:lip},
whereas (b) and (c) follow from \cite[Theorem~2.56 and (2.41)]{AFP}.

In order to simplify the notation, in the following we set
$u^-(x) := \ut(x)$ if $x\in\Omega\setminus S_u$.

Let us choose a function \(\varphi\in C^{\infty}_c(\R^N)\), with support contained
in \(B_1(0)\), such that \(0\leq \varphi \leq 1\).
For every \(\varepsilon > 0\) let \(\varphi_{\varepsilon}(y) := \varphi\left(\frac{y-x}{\varepsilon}\right)\).

By the very definition of normal trace, the following equality holds for
every \(\varepsilon > 0\) small enough:
\begin{equation}\label{f:tra}
\begin{split}
\frac{1}{\varepsilon^{N-1}} &
\int_{\partial\Omega_i} 
[\text{Tr}(\v, \partial\Omega_i) -
\text{Tr}(\B(\cdot, u^-(x)), \partial\Omega_i)]
\, \varphi_{\varepsilon}(y)\, d\mathcal{H}^{N-1}(y)
\\ = {} &
\frac{1}{\varepsilon^{N-1}}
\int_{\Omega_i} \nabla\varphi_{\varepsilon}(y) \cdot [\v(y) - \B(y, u^-(x))]\, dy
\\ & + 
\frac{1}{\varepsilon^{N-1}}
\int_{\Omega_i} \varphi_{\varepsilon}(y) \, d[\Div\v - \Div_x\B(\cdot, u^-(x))](y)\,.
\\ =:{} &  
I_1(\varepsilon) + I_2(\varepsilon)\,.
\end{split}
\end{equation}
Using the change of variable \(z = (y-x)/\varepsilon\),
as \(\varepsilon \to 0\) the left hand side of this equality converges to
\[
[\Trm{\v}{\Sigma}(x) - \beta^-(x, u^-(x))] \int_{\Pi_x} \varphi(z)\, d\mathcal{H}^{N-1}(z)\,,
\]
where \(\Pi_x\) is the tangent plane to \(\Sigma_i\) at \(x\).
Clearly \(\varphi\) can be chosen in such a way that
\(\int_{\Pi_x} \varphi\, d\H^{N-1} > 0\).

In order to prove \eqref{f:tracesv} for \(\Tr^-\) it is then enough
to show that the two integrals \(I_1(\varepsilon)\) and \(I_2(\varepsilon)\)
at the right hand side of \eqref{f:tra} converge to \(0\)
as \(\varepsilon \to 0\).

With the change of variables \( z = (y-x) / \varepsilon\) and by the
very definition of \(\v\) we have that
\[
I_1(\varepsilon) =
\int_{\Omega_i^\varepsilon} \nabla\varphi(z) \cdot
[\B(x + \varepsilon z, u(x+\varepsilon z)) - \B(x + \varepsilon z, u^-(x))]\, dz,
\]
where
\[
\Omega_i^\varepsilon := \frac{\Omega_i - x}{\varepsilon}.
\]
As \(\varepsilon\to 0\), these sets locally converge to the half space 
\(P_x := \{z\in\R^N:\ \pscal{z}{\nu(x)} < 0\}\),
hence
\[
\lim_{\varepsilon\to 0}
\int_{\Omega_i^\varepsilon\cap B_1} |u(x+\varepsilon z) - u^-(x)|\, dz  = 
\lim_{\varepsilon\to 0}
\int_{P_x \cap B_1} |u(x+\varepsilon z) - u^-(x)|\, dz  = 0
\]
(see \cite[Remark 3.85]{AFP}) and, by (ii),
\[
|I_1(\varepsilon)| \leq
\|\bsmall\|_\infty\, \|\nabla\varphi\|_{\infty} 
\int_{\Omega_i^\varepsilon\cap B_1} |u(x+\varepsilon z) - u^-(x)|\, dz  
\to 0.
\]

\smallskip
From (b) we have that
\[
\lim_{\varepsilon\to 0} \frac{1}{\varepsilon^{N-1}}
\left|\int_{\Omega_i}
\varphi_{\varepsilon}(y)\, d \Div_x\B(\cdot, u^-(x))(y)
\right|
\leq
\limsup_{\varepsilon\to 0} \frac{\sigma(B_\varepsilon(x))}{\varepsilon^{N-1}}
=0.
\]
In a similar way, using (c), we get
\[
\lim_{\varepsilon\to 0}
\frac{1}{\varepsilon^{N-1}}
\left|\int_{\Omega_i} \varphi_{\varepsilon} \, d\Div\v\right| = 0,
\]
so that \(I_2(\varepsilon)\)
vanishes as \(\varepsilon\to 0\).

\smallskip
The proof of \eqref{f:tracesv} for \(\Tr^+\) is entirely similar.
\end{proof}

\bigskip

\begin{proof}[Proof of Theorem \ref{chainb5}]
We shall divide the proof into several steps.	

\medskip	
STEP 1. 
We are going to prove that
\[
\mu^{ac}=\pscal{\bsmall(x, u(x))}{\nabla u(x)}\, \mathcal{L}^N.
\]
Let us choose $x\in\R^N$ such that
\begin{itemize}
	\item [(a)] 
there exists the limit
\(\displaystyle\lim_{r\to 0}\frac{\mu(B_r(x))}{r^N}\,;\)
\item[(b)]
\(\displaystyle\lim_{r\to 0}\frac{|D^su|(B_r(x))}{r^N}=0;\)
\item[(c)]
\(\displaystyle
\lim_{r\to 0}\frac1{r^N}\int_{B_r(x)}|
\pscal{\bsmall(y,u(y))}{\nabla u(y)}-
\pscal{\bsmall(x,u(x))}{\nabla u(x)}
|\, dx=0.\)
\end{itemize}
We remark that these conditions are satisfied for $\mathcal L^N$-a.e.\ $x\in\R^N$.
 
Let $r>0$ be such that
\begin{equation}
\label{gtgt}
|D^su|\left(\partial B_r(x)\right)=0.
\end{equation}
Observe that
\[
\nabla u_ \varepsilon = \rho_ \varepsilon \ast Du=
\rho_ \varepsilon \ast \nabla u + \rho_ \varepsilon \ast D^s u.
\]
Hence for every $\phi\in C_0(\R^N)$
with support in \(B_r(x)\) it holds
\begin{equation}
\begin{split}
&\left|\frac1{r^N}\int_{B_r(x)}\phi(y)[\pscal{\bsmall(y,u_\varepsilon(y))}{(\rho_ \varepsilon \ast Du)(y)}-
\pscal{\bsmall(x,u(x))}{\nabla u(x)}]\,dy\right|\\
& \leq
\frac1{r^N}\int_{B_r(x)}\phi(y)\left|\pscal{\bsmall(y,u_\varepsilon(y))}{(\rho_ \varepsilon \ast \nabla u)(y)}-
\pscal{\bsmall(x,u(x))}{\nabla u(x)}\right|\,dy\\
& \quad +\frac1{r^N}\|\phi\|_\infty \|\bsmall\|_\infty\int_{B_r(x)}\rho_ \varepsilon \ast |D^su|\,dy,
\end{split}
\end{equation}
where in the last inequality we use that 
$\left|\rho_ \varepsilon \ast D^su\right|\leq \rho_ \varepsilon \ast |D^su|$.
We note that by \eqref{gtgt}
\[
\lim_{\varepsilon\to 0}\int_{B_r(x)}\rho_ \varepsilon \ast |D^su|\,dy= |D^su| (B_r(x)).
\]
Hence taking the limit as $\varepsilon\to 0$ we obtain
\[
\begin{split}
&\left|\frac1{r^N}\int_{B_r(x)}\phi(y)\ d\mu(y)-
\frac1{r^N}\int_{B_r(x)}\phi(y)  \pscal{\bsmall(x,u(x))}{\nabla u(x)}\,dy\right|\\
& \leq
\frac1{r^N}\int_{B_r(x)}\phi(y)\left|\pscal{\bsmall(y,u(y))}{\nabla u(y)}-
\pscal{\bsmall(x,u(x))}{\nabla u(x)}\right|\,dy\\
& \quad
+\frac1{r^N}\|\phi\|_\infty \|\bsmall\|_\infty \, |D^s u|(B_r(x)).
\end{split}
\]
When $\phi(y)\to 1$ in $B_r(x)$, with \(0\leq \phi \leq 1\), we get
\[
\begin{split}
&\left|\frac1{\omega_N r^N}\mu(B_r(x))-
\pscal{\bsmall(x,u(x))}{\nabla u(x)}\right|\\
& \leq
\frac1{\omega_N r^N}\int_{B_r(x)}\left|\pscal{\bsmall(y,u(y))}{\nabla u(y)}-
\pscal{\bsmall(x,u(x))}{\nabla u(x)}\right|\,dy\\
&\quad
+\frac1{\omega_N r^N} \|\bsmall\|_\infty\, |D^s u|(B_r(x)).
\end{split}
\]
Now the conclusion is achieved by taking the limit for $r\to 0$ and using (b) and (c)
above.

\medskip

STEP 2.
For the jump part of the measure $\mu$ it holds:
\begin{equation}\label{f:muj}
\mu^j=
	[\beta^*(x,u^+)
	-\beta^*(x,u^-)]
\hh\res J_u 
\end{equation}
Namely, this is a direct consequence of Proposition~\ref{p:traces}
in the particular case $\Sigma = J_u$.

\medskip
STEP 3.
From now to the end of the proof,
we shall assume that the additional assumption
\eqref{f:assumpDc} holds.

Let
$S := \bigcup_{q\in Q} S_{\bsmall(\cdot, q)}$.
By assumption \eqref{f:assumpDc} we have that $|D^c u|(S) = 0$.

We claim that, for every $x\in\R^N\setminus S$ and every $t\in\R$, there exists the approximate limit
of $\bsmall$ at $(x,t)$ and
\begin{equation}\label{f:limqj}
\widetilde{\bsmall}(x, t) =
\lim_j \widetilde{\bsmall}(x, q_j),
\qquad
\forall (q_j)\subset Q,\
q_j \to t.
\end{equation}
Namely, 
let us fix a point $x\in\R^N\setminus S$.
By assumptions (i), (ii) and the Dominated Convergence Theorem,
the map
\[
\psi(t) := \lim_{r\to 0} \mean{B_r(x)} \bsmall(y,t)\, dy,
\qquad t\in\R,
\]
is continuous, and $\psi(q) = \widetilde{\bsmall}(x,q)$ for every $q\in Q$.
Hence the limit in \eqref{f:limqj} exists for every $t$ and
it is independent of the choice of the sequence $(q_j)\subset Q$
converging to $t$.

Let $t\in\R$ be fixed, let us denote by $c\in\R^N$ the value
of the limit in \eqref{f:limqj}
and let us prove that $c = \widetilde{\bsmall}(x,t)$.
We have that
\[
\mean{B_r(x)} |\bsmall(y,t) - c|\, dy\leq
\mean{B_r(x)} |\bsmall(y,t) - \bsmall(y, q_j)|\, dy
+ \mean{B_r(x)} |\bsmall(y, q_j) - \widetilde{\bsmall}(x, q_j)|\, dy
+ |\widetilde{\bsmall}(x, q_j) - c|.
\]
As $j\to +\infty$,
the first integral at the r.h.s.\ converges to $0$
by (i), (ii) and the Dominated Convergence Theorem.
The second integral converges to $0$ since $x\in \R^N\setminus S$ and $(q_j)\subset Q$.
Finally,
$\lim_j |\widetilde{\bsmall}(x, q_j) - c| = 0$
by the very definition of $c$,
so that the claim is proved.

\medskip
STEP 4. 
We are going to prove that
\[
\mu^{d}=\pscal{\widetilde{\bsmall}(x, \ut(x))}{D^d u}
\]
in the sense of measures.
We remark that, by Step 3, the approximate limit $\widetilde{\bsmall}(x, t)$
exists for every $(x,t) \in (\R^N\setminus S)\times \R$, with $|D^c u|(S) = 0$.
As a consequence, the function
$x\mapsto \widetilde{\bsmall}(x, \ut(x))$ is well-defined
for $|D^d u|$-a.e.\ $x\in\R^N$,
and it belongs to $L^\infty(\R^N, |D^d u|)$.

If we consider the polar decomposition \(D^d u = \theta\, |D^d u|\),
this equality is equivalent to
\[
\frac{d\mu}{d|D^du|}(x) = \frac{d\mu^{d}}{d|D^du|}(x)=
\pscal{\widetilde{\bsmall}(x, \ut(x))}{\theta(x)}
\]
for $|D^d u|$-a.e.\ $x\in \R^N$.
Let us choose $x\in \R^N$ such that
\begin{itemize}
	\item[(a)]
	\(x\) belongs to the support of \(D^d u\), that is,
	\(|D^d u|(B_r(x)) > 0\) for every \(r >0\);
	\item[(b)] 
	there exists the limit
	\(\displaystyle
	\lim_{r\to 0}\frac{\mu^d(B_r(x))}{|D^d u|(B_r(x))};
	\)
	\item[(c)]
	\(\displaystyle
	\lim_{r\to 0}\frac{|D^j u|(B_r(x))}{|D u|(B_r(x))}=0;
	\)
	\item[(d)]
	\(\displaystyle
	\lim_{r\to 0}\frac1{|D^du|(B_r(x))}\int_{B_r(x)}\left|
	\pscal{\widetilde{\bsmall}(y, \ut(y))}{\theta(y)}-
	\pscal{\widetilde{\bsmall}(x, \ut(x))}{\theta(x)}
	\right|\,d|D^d u|(y)=0.
	\)
\end{itemize}
We remark that these conditions are satisfied
for $|D^d u|$-a.e.\ $x\in\R^N$. 
In particular, (d) follows from the fact that
the map
$y\mapsto \widetilde{\bsmall}(y, \ut(y))$
belongs to $L^\infty(\R^N, |D^d u|)$.

Let $r>0$ be such that
\begin{equation}
\label{gtgt8}
|D^ju|\left(\partial B_r(x)\right)=0.
\end{equation}
Observe that
\(\nabla u_ \varepsilon = \rho_ \varepsilon \ast Du=
\rho_ \varepsilon \ast D^d u+\rho_ \varepsilon \ast D^ju\).
Hence, for every $\phi\in C_c(\R^N)$ with support in \(B_r(x)\), it holds
\begin{equation}\label{f:diseq}
\begin{split}
&\Bigg|\frac1{|D^d u|(B_r(x))}\int_{B_r(x)}\phi(y)
\pscal{\bsmall(y,u_\varepsilon(y))}{(\rho_ \varepsilon \ast Du)(y)}\,dy
\\ & \quad -
\frac1{|D^d u|(B_r(x))}\int_{B_r(x)}\phi(y)  
\pscal{\widetilde{\bsmall}(x, \ut(x))}{\theta(x)}\,d|D^d u|(y)\Bigg|\\
& \leq
\Bigg|\frac1{|D^d u|(B_r(x))}\int_{B_r(x)}\phi(y) 
\pscal{\bsmall (y,u_\varepsilon(y))}{(\rho_ \varepsilon \ast D^d u)(y)}\, dy
\\ & \quad -
\frac1{|D^d u|(B_r(x))}\int_{B_r(x)}\phi(y) 
\pscal{\widetilde{\bsmall} (x,\ut(x))}{\theta(x)}\,d|D^d u|(y)\Bigg|
\\ & \quad 
+ \frac1{|D^d u|(B_r(x))}\|\phi\|_\infty \|\bsmall\|_\infty\int_{B_r(x)}\rho_ \varepsilon \ast |D^ju|\,dy,
\end{split}
\end{equation}
where in the last inequality we use that 
$\left|\rho_ \varepsilon \ast D^ju\right|\leq \rho_ \varepsilon \ast |D^ju|$.
We note that by \eqref{gtgt8}
\[
\lim_{\varepsilon\to 0}\int_{B_r(x)}\rho_ \varepsilon \ast |D^ju|\,dy= 
|D^j u| (B_r(x)).
\]
Hence by taking the limit as $\varepsilon\to 0$ in \eqref{f:diseq} we obtain
\[
\begin{split}
&\Bigg|\frac1{|D^d u|(B_r(x))}\int_{B_r(x)}\phi(y)\ d\mu(y)\\&-
\frac1{|D^d u|(B_r(x))}\int_{B_r(x)}\phi(y)
\pscal{\widetilde{\bsmall}(x,\ut(x))}{\theta(x)}\,d|D^d u|(y)\Bigg|\\
&\leq 
\frac1{|D^d u|(B_r(x))}\int_{B_r(x)}\phi(y)
\left|\pscal{\widetilde{\bsmall}(y,\ut(y))}{\theta(y)}-
\pscal{\widetilde{\bsmall}(x,\ut(x))}{\theta(x)}\,d|D^d u|(y)\right|\\
&\quad +
\frac1{|D^d u|(B_r(x))}\|\phi\|_\infty \|\bsmall\|_\infty\,|D^j u|(B_r(x)).
\end{split}
\]
When $\phi(y)\to 1$ in $B_r(x)$, with \(0\leq \phi \leq 1\), we get
\begin{equation}
\begin{split}
&\left|\frac1{|D^d u|(B_r(x))}\mu(B_r(x))- 
\pscal{\widetilde{\bsmall}(x,\ut(x))}{\theta(x)}\right|\\
& \leq
\frac1{|D^d u|(B_r(x))}\int_{B_r(x)}\left|
\pscal{\widetilde{\bsmall}(y,\ut(y))}{\theta(y)}-
\pscal{\widetilde{\bsmall}(x,\ut(x))}{\theta(x)}\right|\, d|D^d u|(y)\\
& \quad +
\frac1{|D^d u|(B_r(x))} \|\bsmall\|_\infty\,
|D^ju|(B_r(x)).
\end{split}
\end{equation}
The conclusion is achieved now by taking $r\to 0$ and by using (c) and (d).
\end{proof}

\section{Gluing constructions and extension theorems}
\label{s:gluing}

A direct consequence of 
Theorems~\ref{chainb4}, \ref{chainb5} 
and \cite[Theorems 5.1 and 5.3]{ComiPayne}
are the following
gluing constructions and extension theorems
(the proofs are entirely similar to that of \cite[Theorems 8.5 and 8.6]{ChToZi} and \cite[Theorems 5.1 and 5.3]{ComiPayne}).

\begin{theorem}[Extension]
	Let $W \Subset \textrm{int}(E) \subset E \Subset U \subset\Omega$,
	where $\Omega, U, W\subset\R^N$ are open sets and $E$ is a set of finite perimeter in $\Omega$.
	Let 
	\[
	\bsmall_1\colon U\times\R\to\R^N,
	\qquad
	\bsmall_2\colon (\Omega\setminus\overline{W})\times\R\to\R^N
	\]
	satisfy assumptions (i)--(iv) in $U\times\R$ and $\Omega\setminus\overline{W}$ respectively.
	Let $\B_1, \B_2$ be the corresponding integral functions with respect to the second variable.
	Given $u_1\in\BVLloc[U]$ and $u_2\in\BVLloc[\Omega\setminus\overline{W}]$,
	let $\v_i(x) := \B_i(x, u_i(x))$, $i = 1,2$. 
	Then the function
	\[
	\v(x) :=
	\begin{cases}
	\v_1(x),
	& \text{if}\ x\in E,\\
	\v_2(x),
	& \text{if}\ x\in\Omega\setminus E,
	\end{cases}
	\]
	belongs to $\DMloc$ and
	\[
	\Div\v = \chi_{E^1} \Div\v_1 + \chi_{E^0} \Div\v_2 +
	[\Trp{\v_1}{\partial^* E} - \Trm{\v_2}{\partial^* E}]\,
	\H^{N-1}\res\partial^*E.  
	\]
\end{theorem}

\begin{theorem}[Gluing]\label{t:gluing}
Let $U\Subset\Omega\subset\R^N$ be open sets with $\H^{N-1}(\partial U) < \infty$,
and let 
\[
\bsmall_1\colon U\times\R\to\R^N,
\qquad
\bsmall_2\colon (\Omega\setminus\overline{U})\times\R\to\R^N
\]
satisfy assumptions (i)--(iv) in $U\times\R$ and $\Omega\setminus\overline{U}$ respectively.
Let $\B_1, \B_2$ be the corresponding integral functions with respect to the second variable.
Given $u_1\in BV(U)\cap L^\infty(U)$ and 
$u_2\in BV(\Omega\setminus\overline{U})\cap L^\infty(\Omega\setminus\overline{U})$,
let 
\[
\v_1(x) := \begin{cases}
\B_1(x, u_1(x)),
&\text{if}\ x\in U,\\
0,
&\text{if}\ x\in\Omega\setminus{U},
\end{cases}
\qquad
\v_2(x) := \begin{cases}
0,
&\text{if}\ x\in \overline{U},\\
\B_2(x, u_2(x)),
&\text{if}\ x\in\Omega\setminus\overline{U}.
\end{cases}
\]
Then $\v_1, \v_2, \v \in \DM(\Omega)$ and
\[
\Div\v = \chi_{U^1} \Div\v_1 + \chi_{U^0} \Div\v_2 +
[\Trp{\v_1}{\partial^* U} - \Trm{\v_2}{\partial^* U}]\,
\H^{N-1}\res\partial^*U.  
\]
\end{theorem}

\section{The Gauss--Green formula}
\label{s:green}

Let $E\Subset\Omega$ be a set of finite perimeter.
Using the conventions of Section~\ref{distrtraces},
we will assume that the generalized normal vector on \(\partial^* E\) coincides
\(\hh\)-a.e.\ on \(\partial^* E\) with the measure--theoretic 
\textsl{interior} unit normal vector \(\nuint_E\) to \(E\).

We recall that, if $u\in BV_{\rm loc}(\Omega)$, then we will understand
$u^\pm(x) = \ut(x)$ for every $x\in \Omega\setminus S_u$.

The following result has been proved in \cite{CDC3} in the case
$\B(x,w) = w\, \A(x)$
(see also \cite{ComiPayne,LeoSar} for related results).
To simplify the notation, we will denote by
$\mu := (\bsmall(\cdot,u), Du)$ 
the Radon measure introduced in~\eqref{f:bxu}.

\begin{theorem}[Gauss--Green formula]\label{t:GG}
	Let $\bsmall$
	satisfy assumptions (i)--(iv)
	and let \(\B\) be defined by \eqref{f:B}.
	Let \(E \Subset \Omega\) be a bounded set with finite perimeter
	and let  $u\in \BVLloc$.
	Then the following Gauss--Green formulas hold:
	\begin{gather}
	\int_{E^1} \frac{ F(x, u^+(x)) + F(x, u^-(x)) }{2}\, d\sigma(x)
	+ \mu(E^1) =
	-\int_{\partial ^*E} \beta^+(x,u^+(x)) \ d\mathcal H^{N-1}\,,\label{GreenIB}\\
	\begin{split}
	\int_{E^1\cup\partial^*E} \frac{ F(x, u^+(x)) + F(x, u^-(x)) }{2}\, d\sigma(x)
	& + \mu(E^1\cup\partial^*E) \\
	& = -\int_{\partial ^*E} \beta^-(x,u^-(x)) \ d\mathcal H^{N-1}\,,
	\end{split}\label{GreenIB2}
	\end{gather}
	where $E^1$ is the measure theoretic interior of $E$,
	and $\beta^\pm(\cdot,t) := \Trace{\B(\cdot,t)}{\partial^* E}$
	are the normal traces of \(\B(\cdot, t)\) when \(\partial^* E\) is oriented
	with respect to the interior unit normal vector.
\end{theorem}

\begin{proof}
Since $E$ is bounded we can assume, without loss of generality,
that $u\in BV(\R^N)\cap L^\infty(\R^N)$.
By Theorem~\ref{chainb4}, the composite function $\v(x) := \B(x, u(x))$
belongs to $\DM$.
Since $E$ is a bounded set of finite perimeter, the characteristic function
$\chi_E$ is a compactly supported $BV$ function, so that
$\Div (\chi_E \v) (\R^N) = 0$
(see \cite[Lemma~3.1]{ComiPayne}).

We recall that, for every $w \in BV\cap L^\infty$ and every $\A\in \DM$, 
it holds
\[
\Div(w \A) = w^* \Div \A + (\A, Dw),
\]
where $(\A, Dw)$ is the Anzellotti pairing between the function $w$ and the
vector field $\A$
(see \eqref{f:anz}).
Hence, using the above formula with $w = \chi_E$ and $\A = \v$,
it follows that
\begin{equation}\label{f:eqgg1}
0 = \Div (\chi_E \v) (\R^N)
= \int_{\R^N} \chi_E^* \, d\Div \v + (\v, D\chi_E)(\R^N).
\end{equation}
Since
\[
(\v, D\chi_E) =
[\Trp{\v}{\partial^*E} - \Trm{\v}{\partial^* E}] 
\H^{N-1}\res \partial^*E,
\]
from 
Proposition~\ref{p:traces} we get
\begin{equation}\label{f:eqgg2}
(\v, D\chi_E)(\R^N) = \int_{\partial^* E} 
[\beta^+(x, u^+(x)) - \beta^-(x, u^-(x))]
\, d\H^{N-1}(x).
\end{equation}
Since $\chi^*_{E} = \chi_{E^1} + \frac{1}{2} \chi_{\partial^* E}$, using again 
Proposition~\ref{p:traces} and
\eqref{lolo} it holds
\begin{equation}\label{f:eqgg3}
\begin{split}
\int_{\R^N} \chi_E^* \, d\Div \v = {} &
\Div \v (E^1) + \frac{1}{2} \int_{\partial^* E} 
\frac{[\beta^+(x, u^+(x)) + \beta^-(x, u^-(x)}{2}\,
d\H^{N-1}(x)
\\ = {} &
\int_{E^1} \frac{ F(x, u^+(x)) + F(x, u^-(x)) }{2}\, d\sigma(x)
+ \mu(E^1)
\\ & 
+ \frac{1}{2} \int_{\partial^* E} 
\frac{[\beta^+(x, u^+(x)) + \beta^-(x, u^-(x)}{2}\,
d\H^{N-1}(x)\,.
\end{split}
\end{equation}
Formula \eqref{GreenIB} now follows from
\eqref{f:eqgg1}, \eqref{f:eqgg2} and \eqref{f:eqgg3}.

The proof of \eqref{GreenIB2} is entirely similar.
\end{proof}

It is worth to mention a consequence of the gluing construction 
given in Theorem~\ref{t:gluing} and the Gauss--Green formula \eqref{GreenIB}.
To this end, following \cite{LeoSar2}, any bounded open
set $\Omega\subset\R^N$ with finite perimeter, such that
$\H^{N-1}(\partial\Omega) = \H^{N-1}(\partial^*\Omega)$,
will be called \textsl{weakly regular}.
For weakly regular sets we have the following version of
the Gauss--Green formula
(see \cite[Corollary 5.5]{ComiPayne} for a similar statement
for autonomous vector fields).

\begin{theorem}[Gauss--Green formula for weakly regular sets]\label{t:GGw}
Let $\Omega\subset\R^N$ be a weakly regular set.
Let $\bsmall$ satisfy assumptions (i)--(iv),
let \(\B\) be defined by \eqref{f:B}
and let  $u\in BV(\Omega)\cap L^\infty(\Omega)$.
Then the following Gauss--Green formula holds:
\begin{equation}
\label{f:GGw}
\int_{\Omega} \frac{ F(x, u^+(x)) + F(x, u^-(x)) }{2}\, d\sigma(x)
+ \mu(\Omega) =
-\int_{\partial\Omega} \beta^+(x,u^+(x)) \ d\mathcal H^{N-1}\,.
\end{equation}
\end{theorem}

\begin{proof}
Since $\Omega$ is a set of finite perimeter,
it holds $\partial^*\Omega \subseteq \partial\Omega$,
hence the assumption $\H^{N-1}(\partial\Omega) = \H^{N-1}(\partial^*\Omega)$
of weak regularity implies that
$\H^{N-1}(\partial\Omega\setminus\partial^*\Omega) = 0$.
Consequently,
\begin{equation}\label{f:incl}
\H^{N-1}\res \partial\Omega = \H^{N-1}\res \partial^*\Omega,
\quad
\H^{N-1}(\Omega^1\setminus\Omega) = 0,
\quad
\H^{N-1}(\Omega^0\setminus(\R^N\setminus\overline{\Omega})) = 0.
\end{equation}

Let us consider the vector field
\[
\v(x) :=
\begin{cases}
\v_1(x) := \B(x, u(x)), 
&\text{if}\ x\in\Omega,\\
0,
&\text{if}\ x\in \R^N\setminus\Omega\,.
\end{cases}
\]
By Theorem \ref{t:gluing} and \eqref{f:incl} we have that $\v\in\DM$ and
\[
\Div\v = \chi_{\Omega} \Div\v_1 + \Trp{\v_1}{\partial\Omega}\,\H^{N-1}\res\partial\Omega,
\]
hence \eqref{f:GGw} follows reasoning as in the proof of \eqref{GreenIB}.
\end{proof}

\def\cprime{$'$}

\end{document}